\documentclass[a4paper,10pt,reqno]{amsart}

\usepackage[english]{babel}
\usepackage[utf8]{inputenc}
\usepackage{amsmath,amssymb,amsfonts,amsthm,amscd}
\usepackage[mathscr]{eucal}
\usepackage{color}
\usepackage{enumitem}
\setlist[itemize]{topsep=0pt, leftmargin=2em}
\setlist[enumerate]{topsep=0pt, leftmargin=2em}

\usepackage{hyperref}
\usepackage{graphicx}

\DeclareMathOperator{\arctanh}{arc\,tanh}
\DeclareMathOperator{\arcsinh}{arc\,sinh}
\DeclareMathOperator{\area}{area}

\numberwithin{equation}{section}
\newtheorem{theorem}{Theorem}
\newtheorem*{theorem*}{Theorem}

\newtheorem{lemma}{Lemma}

\theoremstyle{definition}

\theoremstyle{remark}
	\newtheorem{remark}{Remark}

\title[Genus one minimal $k$-noids and saddle towers in $\mathbb{H}^2\times\mathbb{R}$]{Genus one minimal $k$-noids\\ and saddle towers in $\mathbb{H}^2\times\mathbb{R}$}
\date{}

\author{Jesús Castro-Infantes}
\address{Departamento de Geometría y Topología\\
Universidad de Granada, SPAIN}
\email{jcastroinfantes@ugr.es}

\author{José M.\ Manzano}
\address{Departamento de Matemáticas\\
Universidad de Jaén, SPAIN}
\email{jmprego@ujaen.es}

\subjclass[2010]{Primary 53A10; Secondary 53C30}

\keywords{Minimal surfaces, finite total curvature, minimal $k$-noids, saddle towers, conjugate construction}

\begin{document}

\begin{abstract}
For each $k\geq 3$, we construct a 1-parameter family of complete properly Alexandrov-embedded minimal surfaces in the Riemannian product space $\mathbb{H}^2\times\mathbb{R}$ with genus $1$ and $k$ embedded ends asymptotic to vertical planes. We also obtain complete minimal surfaces with genus $1$ and $2k$ ends in the quotient of $\mathbb{H}^2\times\mathbb{R}$ by an arbitrary vertical translation. They all have dihedral symmetry with respect to $k$ vertical planes, as well as finite total curvature $-4k\pi$. Finally, we also provide examples of complete properly Alexandrov-embedded minimal surfaces with finite total curvature with genus $1$ in quotients of $\mathbb{H}^2\times\mathbb{R}$ by the action of a hyperbolic or parabolic translation.
\end{abstract}

\maketitle

\section{Introduction}

The theory of complete minimal surfaces in $\mathbb{H}^2\times\mathbb{R}$ with finite total curvature, i.e., those whose Gauss curvature is integrable, has received considerable attention during the last decade, mainly triggered by Collin and Rosenberg~\cite{CR}. The combined work of Hauswirth, Nelli, Sa Earp and Toubiana~\cite{HNST}, and Hauswirth, Menezes and Rodríguez~\cite{HMR} shows that a complete minimal surface immersed in $\mathbb{H}^2\times\mathbb{R}$ has finite total curvature if and only if it is proper, has finite topology and each of its ends is asymptotic to an admissible polygon, i.e., a curve homeomorphic to $\mathbb{S}^1$ consisting of finitely-many alternating complete vertical and horizontal ideal geodesics, see~\cite{HMR}. Here the product compactification of $\mathbb{H}^2\times\mathbb{R}$ is considered, in which the horizontal (resp.\ vertical) ideal boundary consists of two disks $\mathbb H^2\times\{\pm\infty\}$ (resp.\ the cylinder $\partial_{\infty}\mathbb H^2\times\mathbb{R}$). Ideal horizontal geodesics are those of the form $\Gamma\times\{+\infty\}$ or $\Gamma\times\{-\infty\}$, being $\Gamma$ a geodesic of $\mathbb H^2$, whereas ideal vertical geodesics are those of the form $\{p_\infty\}\times\mathbb{R}$, where $p_\infty\in\partial_\infty\mathbb H^2$ is an ideal point.

Combining the above classification with the previous work of Hauswirth and Rosenberg~\cite{HR}, the following Gauss--Bonnet-type formula for a complete minimal surface $\Sigma$ immersed in $\mathbb{H}^2\times\mathbb{R}$ with finite total curvature holds true:
\begin{equation}\label{eqn:generalized-GB}
\int_\Sigma K=2\pi\chi(\Sigma)-2\pi m=2\pi(2-2g-k-m),
\end{equation}
where $g$ and $k$ are the genus and the number of ends of $\Sigma$, respectively, $\chi(\Sigma)=2-2g-k$ its Euler characteristic, $K$ its Gauss curvature, and $m$ is the total number of horizontal ideal geodesics in $\mathbb{H}^2\times\{+\infty\}$, among all polygonal components associated with the ends of $\Sigma$. Observe that the union of all these components consists of $m$ ideal horizontal geodesics in $\mathbb{H}^2\times\{+\infty\}$, $m$ ideal horizontal geodesics in $\mathbb{H}^2\times\{-\infty\}$, and $2m$ ideal vertical geodesics, so the term $2\pi m$ in~\eqref{eqn:generalized-GB} can be understood as the sum of exterior angles of the asymptotic boundary of $\Sigma$. We also remark that Formula~\eqref{eqn:generalized-GB} has been extended to some quotients of $\mathbb{H}^2\times\mathbb{R}$ by Hauswirth and Menezes~\cite{HM}.

Although this characterization is very satisfactory from a theoretical point of view, it seems tough in general to determine whether or not a given family of admissible polygons actually bounds a minimal surface, or if a given topological type can be realized by such a surface. In fact, there are not many examples of surfaces with finite total curvature in the literature. Let us highlight some of them in terms of the three parameters $(g,k,m)$ appearing in~\eqref{eqn:generalized-GB}:
\begin{itemize}
\item  The simplest case is that of flat minimal surfaces, which must be vertical planes (i.e., of the form $\Gamma\times\mathbb{R}$, being $\Gamma\subset\mathbb{H}^2$ a complete geodesic) because of Gauss equation. In particular, vertical planes are the only complete minimal surfaces with finite total curvature and  $(g,k,m)=(0,1,1)$, see also~\cite[Corollary~5]{HST}. 
\item A minimal Scherk graph in $\mathbb{H}^2\times\mathbb{R}$ is a minimal graph over a geodesic ideal polygon of $\mathbb{H}^2$ with $2a$ vertexes, $a\geq 2$, taking alternating limit values $+\infty$ and $-\infty$ on the sides of the polygon. A characterization of polygons carrying such a surface is analyzed in~\cite{MRR}, in which case they have finite total curvature and satisfy $(g,k,m)=(0,1,a)$. The case $a=2$ gives rise to the only complete minimal surfaces with total curvature $-2\pi$, as shown by Pyo and Rodríguez~\cite[Theorem~4.1]{PR}. We can find as well the \emph{twisted} Scherk minimal surfaces~\cite{PR} with $(g,k,m)=(0,1,2b+1)$, $b\geq 1$, and total curvature $-4b\pi$ that are no longer graphs or bigraphs, some of which are embedded.
\item Minimal $k$-noids constructed by Morabito and Rodríguez~\cite{MorRod} (also by Pyo~\cite{Pyo} in the symmetric case) have finite total curvature, genus $0$ and $k$ ends asymptotic to vertical planes. This gives $g=0$ and $k=m\geq 2$. 
\item Horizontal catenoids are the only complete minimal surfaces immersed in $\mathbb{H}^2\times\mathbb{R}$ with finite total curvature and $k=m=2$, see~\cite{HMR,HNST}. The family of minimal surfaces with finite total curvature and $k=m\geq 3$ is not hitherto well understood, not even in the case $g=0$. The most general construction was given by Martín, Mazzeo and Rodríguez~\cite{MMR}, who found properly embedded minimal surfaces with finite total curvature in $\mathbb{H}^2\times\mathbb{R}$ of genus $g$ and $k$ ends asymptotic to vertical planes (and hence $m=k$), for arbitrary $g\geq 0$ and $k$ arbitrarily large depending on $g$.
\end{itemize}

In this paper we provide highly symmetric examples with $g=1$ and $m=k\geq 3$, which are hence conformally equivalent to a torus with $k$ punctures. They can be thought of as the counterpart in $\mathbb{H}^2\times\mathbb{R}$ of the genus $1$ minimal $k$-noids in $\mathbb{R}^3$ obtained by Mazet~\cite{Maz}. Outside a compact subset, our surfaces look like the minimal $k$-noids in~\cite{MorRod,Pyo}, and they are not globally embedded in general. Notice that there are no such examples with $k=2$ due to the aforesaid uniqueness of horizontal catenoids in~\cite{HNST}. Our main result can be stated as follows:

\begin{theorem}~\label{thm:knoids}
For each $k\geq 3$, there exists a $1$-parameter family of properly Alexandrov-embedded minimal surfaces in $\mathbb{H}^2\times\mathbb{R}$ with genus $1$ and $k$ ends, dihedrally symmetric with respect to $k$ vertical planes and symmetric with respect to a horizontal plane. They have finite total curvature $-4k\pi$ and each of their ends is embedded and asymptotic to a vertical plane. 
\end{theorem}

The construction of these genus $1$ minimal $k$-noids is based on a conjugate technique, in the sense of Daniel~\cite{Dan} and Hauswirth, Sa Earp and Toubiana~\cite{HST}. Conjugation has been a fruitful technique to obtain constant mean curvature surfaces in $\mathbb{H}^2\times\mathbb{R}$ and $\mathbb{S}^2\times\mathbb{R}$, see~\cite{ManTor,ManTor2,MPT,Maz,MRR2,MorRod,Plehnert,Plehnert2,Pyo} and the references therein. We begin by considering a solution to an improper Dirichlet problem~\cite{MRR,NR} in $\mathbb{H}^2\times\mathbb{R}$ over an unbounded geodesic triangle $\Delta\subset\mathbb{H}^2$, a so-called Jenkins--Serrin problem~\cite{MRR}. These solutions are minimal graphs over the interior of $\Delta$ with prescribed finite and infinite values when one approaches $\partial\Delta$. The conjugate surface is another minimal graph in $\mathbb{H}^2\times\mathbb{R}$ whose boundary is made of curves lying on totally geodesic surfaces, i.e., vertical and horizontal planes. Since there are isometric reflections across such planes in $\mathbb{H}^2\times\mathbb{R}$, the conjugate surface can be extended to a complete surface under suitable conditions. In order to prescribe the symmetries stated in Theorem~\ref{thm:knoids}, we will encounter two period problems that will impose further restrictions on $\Delta$ and on the boundary values of the Jenkins--Serrin problem. 

Our conjugate approach is inspired by the genus $1$ minimal $k$-noids in $\mathbb{R}^3$ given by Mazet~\cite{Maz}, and by the mean curvature $\frac{1}{2}$ surfaces in $\mathbb{H}^2\times\mathbb{R}$ given by Plehnert~\cite{Plehnert}. It is important to remark that there exist technical dissimilarities between the cases $H=0$ and $H=\frac{1}{2}$ in $\mathbb{H}^2\times\mathbb{R}$ because of the fact that the conjugate of a surface with mean curvature $\frac{1}{2}$ (resp.\ $0$) is a minimal surface in Heisenberg group $\mathrm{Nil}_3$ (resp.\ $\mathbb{H}^2\times\mathbb{R}$). Furthermore, our construction can be adapted to produce complete minimal surfaces invariant by an arbitrary vertical translation (i.e., in the direction of the factor $\mathbb{R}$), similar to the saddle towers given in~\cite{MorRod}. They have genus $1$ in the quotient and they are not embedded in general.

\begin{theorem}\label{thm:saddle-towers}
For each $k\geq 3$ and each vertical translation $T$, there is a $1$-parameter family of Alexandrov-embedded singly periodic minimal surfaces in $\mathbb{H}^2\times\mathbb{R}$ invariant by $T$ and dihedrally symmetric with respect to $k$ vertical planes and a horizontal plane. They have finite total curvature $-4k\pi$, genus $1$ and $2k$ vertical ends in the quotient of $\mathbb{H}^2\times\mathbb{R}$ by $T$.
\end{theorem}

Our analysis of the period problems will allow us to find surfaces that are not invariant by a discrete group of rotations, but by discrete groups of parabolic or hyperbolic translations, which we will call \emph{parabolic and hyperbolic $\infty$-noids}, respectively. These surfaces have infinitely many ends, and we can guarantee that many of the examples are properly embedded in the hyperbolic case. Although we will not state it explicitly, analogous surfaces can be obtained in the quotient by an arbitrary vertical translation in the spirit of Theorem~\ref{thm:saddle-towers}.

\begin{theorem}\label{thm:infty-noids}
There is a $2$-parameter (resp. $1$-parameter) family of properly embedded (resp. Alexandrov-embedded) minimal surfaces in $\mathbb{H}^2\times\mathbb{R}$ with genus $0$ and infinitely many ends, invariant by a discrete group of hyperbolic (resp. parabolic) translations. Each of their ends is embedded, asymptotic to a vertical plane, and has finite total curvature.
\end{theorem}

The paper is organized as follows: In Section~\ref{sec:preliminaries} we will analyze some aspects of the conjugation of surfaces in $\mathbb{H}^2\times\mathbb{R}$ that will be needed in the construction, and Section~\ref{sec:construction} will be devoted to fill the details of the proof of Theorems~\ref{thm:knoids} and~\ref{thm:saddle-towers}. We will also discuss some open questions about the uniqueness and embeddedness of the constructed surfaces, as well as natural limits of the $1$-parameter family of genus $1$ $k$-noids. In the last part of the paper we will prove Theorem~\ref{thm:infty-noids}.

\medskip
\noindent\textbf{Acknowledgment.} The authors would like to express their gratitude to Magdalena Rodríguez for her valuable comments during the preparation of this manuscript, as well as the anonymous referee for their thorough revision of the manuscript, which has greatly improved the final presentation. This research was supported by \textsc{mineco--feder} project MTM2017-89677-P. The first author is also supported by the FPU program from \textsc{micinn}. The second author is also supported by the \textsc{micinn--feder} project PID2019-111531GA-I00/AEI/10.13039/501100011033.

\section{Preliminaries}\label{sec:preliminaries}

Let $\Sigma$ be a simply connected Riemannian surface. Given an isometric minimal immersion $X:\Sigma\to\mathbb{H}^2\times\mathbb{R}$, Hauswirth, Sa Earp and Toubiana~\cite{HST} proved the existence of another isometric minimal immersion $\widetilde X:\Sigma\to\mathbb{H}^2\times\mathbb{R}$ such that:
\begin{enumerate}
	\item Both immersions induce the same angle function $\nu=\langle N,\partial_t\rangle=\langle\widetilde N,\partial_t\rangle$, where $N$ and $\widetilde N$ stand for unit normal vector fields to $X$ and $\widetilde X$, respectively, and $\partial_t$ is the unit vector field in $\mathbb{H}^2\times\mathbb{R}$ in the direction of the factor $\mathbb{R}$.
	\item The shape operators $S$ and $\widetilde S$ of $X$ and $\widetilde X$, respectively, satisfy $\widetilde S=JS$, where $J$ is the $\frac\pi2$-rotation in $T\Sigma$, chosen such that both $\{\mathrm{d}X_p(u),\mathrm{d}X_p(Ju),N_p\}$ and $\{\mathrm{d}\widetilde X_p(u),\mathrm{d}\widetilde X_p(Ju),\widetilde N_p\}$ are positively oriented bases in $\mathbb H^2\times\mathbb R$ for all non-zero tangent vectors $u\in T_p\Sigma$.
	\item The tangential components $T=\partial_t-\nu N$ and $\widetilde T=\partial_t-\nu\widetilde N$ of $\partial_t$ satisfy $\widetilde X^*\widetilde T=JX^*T$. This implies that $\langle\mathrm{d}X_p(u),\partial_t\rangle=\langle\mathrm{d}\widetilde X_p(Ju),\partial_t\rangle$ for all $u\in T_p\Sigma$.
\end{enumerate}
The immersions $X$ and $\widetilde X$ are called \emph{conjugate} and determine each other up ambient isometries preserving both the global orientation and the vector field $\partial_t$. Our initial surface $X(\Sigma)$ will be a vertical graph over a convex domain, namely a solution of a Jenkins--Serrin problem. This implies that $\widetilde X(\Sigma)$ is also a vertical graph over another (possibly non-convex) domain, due to the Krust-type theorem given by~\cite[Theorem~14]{HST}. Therefore, we can assume that both surfaces are embedded and will use the notation $\Sigma$ and $\widetilde\Sigma$ for the surfaces $X(\Sigma)$ and $\widetilde X(\Sigma)$, respectively.

Although the conjugate surface $\widetilde\Sigma$ is not explicit in general, one can obtain insightful information if the initial surface $\Sigma$ has boundary consisting of horizontal and vertical geodesics intersecting at some vertexes. A curve $\Gamma\subset\Sigma$ is a horizontal (resp.\ vertical) geodesic in $\mathbb{H}^2\times\mathbb{R}$ if and only if the conjugate curve $\widetilde\Gamma\subset\widetilde\Sigma$ lies in a vertical (resp.\ horizontal) totally geodesic surface of $\mathbb{H}^2\times\mathbb{R}$ intersecting $\widetilde\Sigma$ orthogonally along $\widetilde\Gamma$. Furthermore, axial symmetry about $\Gamma$ corresponds to mirror symmetry about $\widetilde\Gamma$, which enables analytic continuation of $\Sigma$ and $\widetilde\Sigma$ across their boundaries. If the angles at the vertexes of $\partial\Sigma$ are integer divisors of $\pi$, then no singularity appears at such vertexes after successive reflections about the boundary components, and both surfaces can be extended to complete (possibly non-embedded) minimal surfaces. We refer to~\cite{ManTor,Plehnert,MRR2} for details.

However, most difficulties concerning the depiction of $\widetilde\Sigma$, and in particular deciding whether or not it is embedded, show up when one tries to understand the behavior of the conjugate of a vertical geodesic. We will now recall some properties on this matter which will be used later in Section~\ref{sec:construction}. Let $\gamma:I\to\partial\Sigma$ be a vertical geodesic with unit speed such that $\gamma'=\partial_t$ (this orientation of vertical geodesics will be fixed throughout the text), where $I\subset\mathbb{R}$ is an interval, and denote by $\widetilde\gamma:I\to\partial\widetilde\Sigma$ the conjugate curve, which will be assumed to lie in $\mathbb{H}^2\times\{0\}$ after a vertical translation. 

Let us consider the half-space model $\mathbb{H}^2\times\mathbb{R}=\{(x,y,t)\in\mathbb{R}^3:y>0\}$, whose metric is given by $y^{-2}(\mathrm{d}x^2+\mathrm{d}y^2)+\mathrm{d}t^2$, with positively oriented orthonormal frame $\{E_1,E_2,\partial_t\}$ given by $E_1=y\partial_x$ and $E_2=y\partial_y$ (observe that $E_1$ is tangent to the foliation of $\mathbb{H}^2$ by horocycles $y=y_0$ with $y_0>0$). Since $\gamma$ is vertical and $\widetilde\gamma$ lies in a horizontal slice, there exist smooth functions $\psi,\theta\in C^\infty(I)$ such that
\begin{align}
 N_{\gamma(t)}&=\cos(\psi(t))E_1+\sin(\psi(t))E_2,\label{eqn:rotation-angle}\\
 \widetilde\gamma'(t)&=\cos(\theta(t))E_1+\sin(\theta(t))E_2,\label{eqn:foliation-angle}
\end{align}
called the angle of rotation of $N$ along $\gamma$ and the angle of rotation of $\widetilde\gamma$ with respect to the foliation by horocycles, respectively. We now collect some relations between these quantities, see also~\cite{CMR,MPT,Plehnert}.

Observe that $E_1$ and $E_2$ are parallel vector fields along $\gamma$ since they satisfy $\overline\nabla_{\partial_t}E_1=\overline\nabla_{\partial_t}E_2=0$, where $\overline\nabla$ stands for the ambient Levi--Civita connection. This is due to the fact that that $\mathbb H^2\times\mathbb{R}$ is a Riemannian product and $E_1$ and $E_2$ do not depend on the variable $t$. By taking derivatives in~\eqref{eqn:rotation-angle}, we get $\overline\nabla_{\gamma'}N=-\psi'\sin(\psi)E_1+\psi'\cos(\psi)E_2=-\psi' N\times\gamma'$, where $\times$ is the cross-product in $\mathbb H^2\times\mathbb{R}$. Using the properties of the conjugation, we deduce the identity
\begin{equation}\label{eqn:kg}
\psi'=-\langle\overline\nabla_{\gamma'}N,N\times\gamma'\rangle=\langle S\gamma',J\gamma'\rangle=-\langle J\widetilde S\widetilde\gamma',J\widetilde\gamma'\rangle=\langle\overline\nabla_{\widetilde\gamma'}\widetilde N,\widetilde\gamma'\rangle=-\kappa_g,
\end{equation}
where $\kappa_g$ is the geodesic curvature of $\widetilde\gamma$ as a curve of $\mathbb{H}^2\times\{0\}$ with respect to the conormal $\widetilde N$ (recall that $\widetilde\Sigma$ intersects $\mathbb{H}^2\times\{0\}$ orthogonally). Now we will obtain further information under the additional assumption that the surfaces are \emph{multigraphs}, i.e., their common angle function $\nu$ has a sign.

\begin{figure}[t]
\begin{center}
\includegraphics[width=0.85\textwidth]{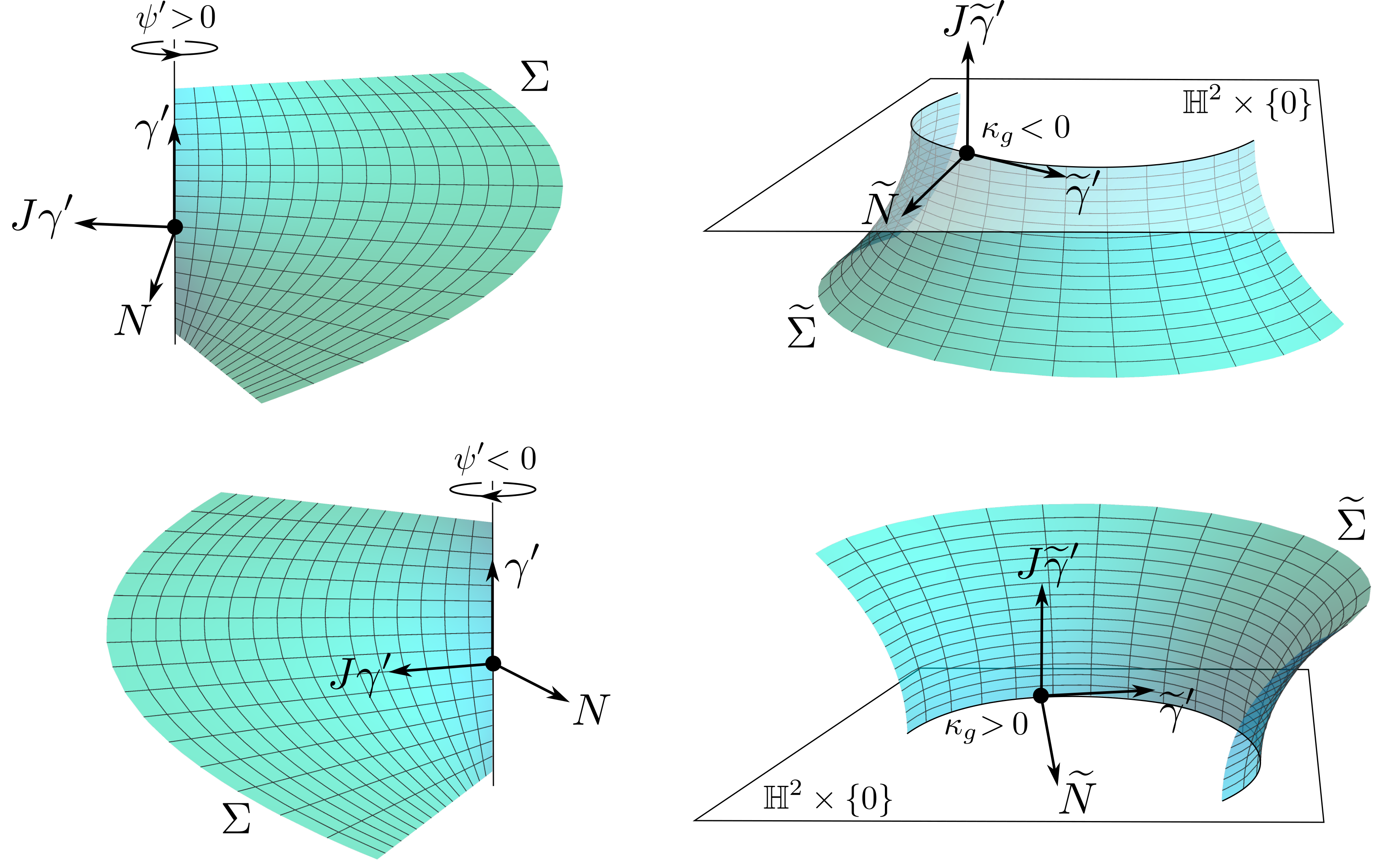}
\caption{Orientation of the conjugate surfaces $\Sigma$ and $\widetilde\Sigma$ according to the direction of rotation of $N$ along a vertical geodesic $\gamma$.}\label{fig:orientation}
\end{center}\end{figure}

\begin{lemma}\label{lem:orientation}
 Assume that the interiors of $\Sigma$ and $\widetilde\Sigma$ are multigraphs over (possibly immersed) domains $\Omega$ and $\widetilde\Omega$, respectively, with angle function $\nu>0$, and let $\gamma$ be a vertical geodesic in $\partial\Sigma$ with $\gamma'=\partial_t$. In the above notation:
\begin{enumerate}[label=\emph{(\alph*)}]
	\item If $\psi'>0$, then $J\gamma'$ (resp.\ $J\widetilde\gamma'=\partial_t$) is a unit outer conormal to $\Sigma$ (resp.\ $\widetilde\Sigma$) along $\gamma$ (resp.\ $\widetilde\gamma$), $\widetilde N$ points to the interior of $\widetilde\Omega$ along $\widetilde\gamma$, and $\widetilde\Sigma$ lies in $\mathbb{H}^2\times(-\infty,0]$ locally around $\widetilde\gamma$ (see Figure~\ref{fig:orientation}, top). 
	\item If $\psi'<0$, then $J\gamma'$ (resp.\ $J\widetilde\gamma'=\partial_t$) is a unit inner conormal to $\Sigma$ (resp.\ $\widetilde\Sigma$) along $\gamma$ (resp.\ $\widetilde\gamma$), $\widetilde N$ points to the exterior of $\widetilde\Omega$ along $\widetilde\gamma$, and $\widetilde\Sigma$ lies in $\mathbb{H}^2\times[0,+\infty)$ locally around $\widetilde\gamma$ (see Figure~\ref{fig:orientation}, bottom). 
\end{enumerate}
Either way, the identity $\theta'=\psi'-\cos(\theta)$ holds true.
\end{lemma}

\begin{proof}
We will only prove item (a) since item (b) is analogous, so we will suppose that $\psi'>0$. As $\{\gamma',J\gamma',N\}$ is positively oriented and $\nu>0$, it follows that $J\gamma'$ points towards the exterior of $\Sigma$ along $\gamma$ (see Figure~\ref{fig:orientation}, top left). Since the rotation $J$ is intrinsic, we deduce that $J\widetilde\gamma'$ points to the exterior of $\widetilde\Sigma$ along $\widetilde\gamma$, and $\widetilde N=\widetilde\gamma'\times J\widetilde\gamma'$ is determined by the ambient orientation.

Assume now by contradiction that $\widetilde N$ points to the exterior of $\widetilde\Omega$ at some point $p$ of $\widetilde\gamma$. Since $\kappa_g=-\psi'<0$ with respect to the conormal $\widetilde N$ and $\nu>0$, we infer that $\widetilde\Sigma$ projects locally into the convex side of $\widetilde\gamma$. This yields a contradiction with the boundary maximum principle by comparing $\widetilde\Sigma$ and a vertical plane tangent to $\widetilde\gamma$ at $p$. Note that $J\widetilde\gamma'$ cannot be equal to $-\partial_t$ (so it must be $J\widetilde\gamma'=\partial_t)$ because it points outside $\widetilde\Sigma$ along $\widetilde\gamma$ and the angle function is positive. As a consequence, a neighborhood of $\widetilde\gamma$ in $\widetilde\Sigma$ is contained in $\mathbb{H}^2\times(-\infty,0]$.

It is easy to calculate $\overline\nabla_{E_1}E_1=E_2$, $\overline\nabla_{E_1}E_2=-E_1$ and $\overline\nabla_{E_2}E_1=\overline\nabla_{E_2}E_2=0$ by using the expressions of $E_1$ and $E_2$ and Koszul formula. On the one hand, this allows us to take derivatives in~\eqref{eqn:foliation-angle} to obtain $\overline\nabla_{\widetilde\gamma'}\widetilde\gamma'=(\theta'+\cos(\theta))(-\sin(\theta)E_1+\cos(\theta)E_2)$. On the other hand, the above discussion shows that $\widetilde N=\widetilde\gamma'\times J\widetilde\gamma'=\widetilde\gamma'\times \partial_t=\sin(\theta)E_1-\cos(\theta)E_2$, so the last identity in the statement follows from plugging these computations in the expression $-\psi'=\kappa_g=\langle\overline\nabla_{\widetilde\gamma'}\widetilde\gamma',\widetilde N\rangle$.
\end{proof}

\section{Construction of genus $1$ saddle towers and $k$-noids}\label{sec:construction}

The first part of this section is devoted to prove Theorems~\ref{thm:knoids} and~\ref{thm:saddle-towers}. The arguments leading to these results are based on a conjugate construction that depends on a parameter $0<l\leq\infty$ that will be fixed henceforth. The case $0<l<\infty$ gives the saddle towers whose fundamental pieces lie in a slab of height $l$, whereas the case $l=\infty$ gives rise to the $k$-noids. Although a limit argument for $l\to\infty$ would imply the latter (as in~\cite{MorRod}) we will discuss both cases together.

\subsection{The conjugate construction}\label{subsec:periods} 

Let $\Delta$ a geodesic triangle with sides $\ell_1$, $\ell_2$ and $\ell_3$, and opposite vertexes $p_1$, $p_2$ and $p_3$. Assume that $\Delta$ is acute and the length of $\ell_2$ is given by the aforesaid parameter $l\in(0,\infty]$, so that $p_1$ is ideal if $l=\infty$. Therefore $\Delta$ is determined by the length $a$ of $\ell_1$ and by the angle $\varphi$ at $p_2$. Given $b\in\mathbb{R}$, consider the unique solution $\Sigma(a,\varphi,b)$ to the Jenkins--Serrin problem over $\Delta$ with boundary values $b$ along $\ell_1$, $+\infty$ along $\ell_2$, and $0$ along $\ell_3$. Existence and uniqueness of such a solution is guaranteed under these boundary conditions, see~\cite{MRR} and the references therein. In particular, the interior of $\Sigma(a,\varphi,b)$ is a minimal graph over $\Delta$ whose boundary consists of two horizontal geodesics $h_1$ and $h_3$ lying in $\mathbb{H}^2\times\{b\}$ and $\mathbb{H}^2\times\{0\}$, respectively, and three vertical geodesics $v_1$, $v_2$ and $v_3$ projecting onto $p_1$, $p_2$ and $p_3$, respectively. Note that $v_1$ is an ideal vertical geodesic provided that $l=\infty$. The boundary of $\Sigma(a,\varphi,b)$ also contains an horizontal ideal geodesic $h_2\subset\mathbb{H}^2\times\{+\infty\}$ projecting onto $\ell_2$. 

Since $\Delta\subset\mathbb{H}^2$ is convex, the conjugate minimal surface $\widetilde\Sigma(a,\varphi,b)\subset\mathbb{H}^2\times\mathbb{R}$ is a graph over some domain $\widetilde\Delta\subset\mathbb{H}^2$ due to the Krust-type theorem in~\cite{HST}. The normal $N$ along $v_2$ or $v_3$ rotates counterclockwise, so theoretical conjugate curves $\widetilde{v}_2$ and $\widetilde{v}_3$ lie in horizontal planes and are convex towards the exterior of $\widetilde\Delta$ by Lemma~\ref{lem:orientation}. This lemma also implies that $\widetilde\Sigma(a,\varphi,b)$ lies locally below the horizontal slices containing $\widetilde{v}_2$ and $\widetilde{v}_3$. We will assume that $\widetilde v_3\subset\mathbb H^2\times\{0\}$ in the sequel. If we decompose $\widetilde h_i=(\beta_i,z_i)\in\mathbb H^2\times\mathbb{R}$, $i\in\{1,2,3\}$, where $\widetilde h_i$ is the conjugate of $h_i$, then it is easy to check that $\|\beta_i'\|=|\nu|$ and $|z_i'|=\sqrt{1-\nu^2}$, so points at which $\nu$ takes the values $0$ or $\pm 1$ will be the key to understand the behavior of $\widetilde h_i$.

\begin{lemma}\label{lem:angle01}
The angle function $\nu$ of $\Sigma(a,\varphi,b)$ is zero precisely at $v_2\cup v_3$ if $l=\infty$ or $v_1\cup v_2\cup v_3$ if $l<\infty$. Furthermore, there is exactly one point of $\Sigma(a,\varphi,b)$ with $\nu=1$ and it belongs to $h_1$.
\end{lemma}

\begin{proof}
The interior of $\Sigma(a,\varphi,b)$ is a graph so there are no interior zeros of $\nu$. Besides, zeros in the interior of $h_i$, $i\in\{1,2,3\}$, would contradict the boundary maximum principle by comparing $\Sigma(a,\varphi,b)$ and the vertical plane $\ell_i\times\mathbb{R}$.

If $\nu(p)=1$ at some interior point $p$, then the intersection of $\Sigma(a,\varphi,b)$ and the horizontal slice containing $p$ is an equiangular set of curves, and two of these curves starting at $p$ must reach some $h_i$ or have a common endpoint at some $v_i$. Either way, we can find a region of $\Sigma(a,\varphi,b)$ bounded by horizontal curves, and a standard application of the maximum principle with respect to horizontal slices gives a contradiction. If $l=\infty$ and the endpoints are in $v_1$, then the region enclosed by the two curves has an end of type 1 in the sense of~\cite[Definition~4.14]{MMR}, and the general maximum principle~\cite[Theorem~4.16]{MMR} applies in this case. By the boundary maximum principle there are no points with $\nu=1$ at $h_2\cup h_3$, but the normal $N$ rotates an angle of $\pi$ along $h_1$ (see Figure~\ref{fig:conjugate-surfaces}), which gives rise to at least one point in $h_1$ with $\nu=1$. If there were more than one, then by intersecting the surface and the horizontal slice containing $h_1$, we would get another region bounded by horizontal curves, which gives a similar contradiction as in the above argument.
\end{proof}

We deduce that $\widetilde h_3$ projects one-to-one into the factors $\mathbb{H}^2$ and $\mathbb{R}$, as well as $\widetilde h_1$ into the factor $\mathbb H^2$. However, the component of $\widetilde h_1$ in the factor $\mathbb R$ has a minimum at the unique point where $\nu=1$. On the other hand, the conjugate curve $\widetilde{h}_2$ can be thought of as an ideal vertical geodesic of length $l$ in $\partial_\infty\mathbb{H}^2\times\mathbb{R}$. Since $\Sigma(a,\varphi,b)$ becomes vertical when one approaches the side $\ell_2$, the length of the ideal vertical geodesic $\widetilde h_2$ is equal to $l$, in particular $\widetilde v_1\subset\mathbb H^2\times\{-l\}$. If $l<\infty$, then $N$ rotates clockwise along $v_1$, so the conjugate curve $\widetilde v_1$ lies in a horizontal slice (which $\Sigma(a,\varphi,b)$ approaches from above) and is convex towards the exterior of $\widetilde\Delta$ by Lemma~\ref{lem:orientation}. If $l=\infty$, then $\widetilde{v}_1$ is an ideal horizontal geodesic in $\mathbb{H}^2\times\{-\infty\}$, see also~\cite[Corollary~2.4]{CMR}. A depiction of the conjugate surfaces is given in Figure~\ref{fig:conjugate-surfaces}.

\begin{figure} 
	\includegraphics[width=0.9\textwidth]{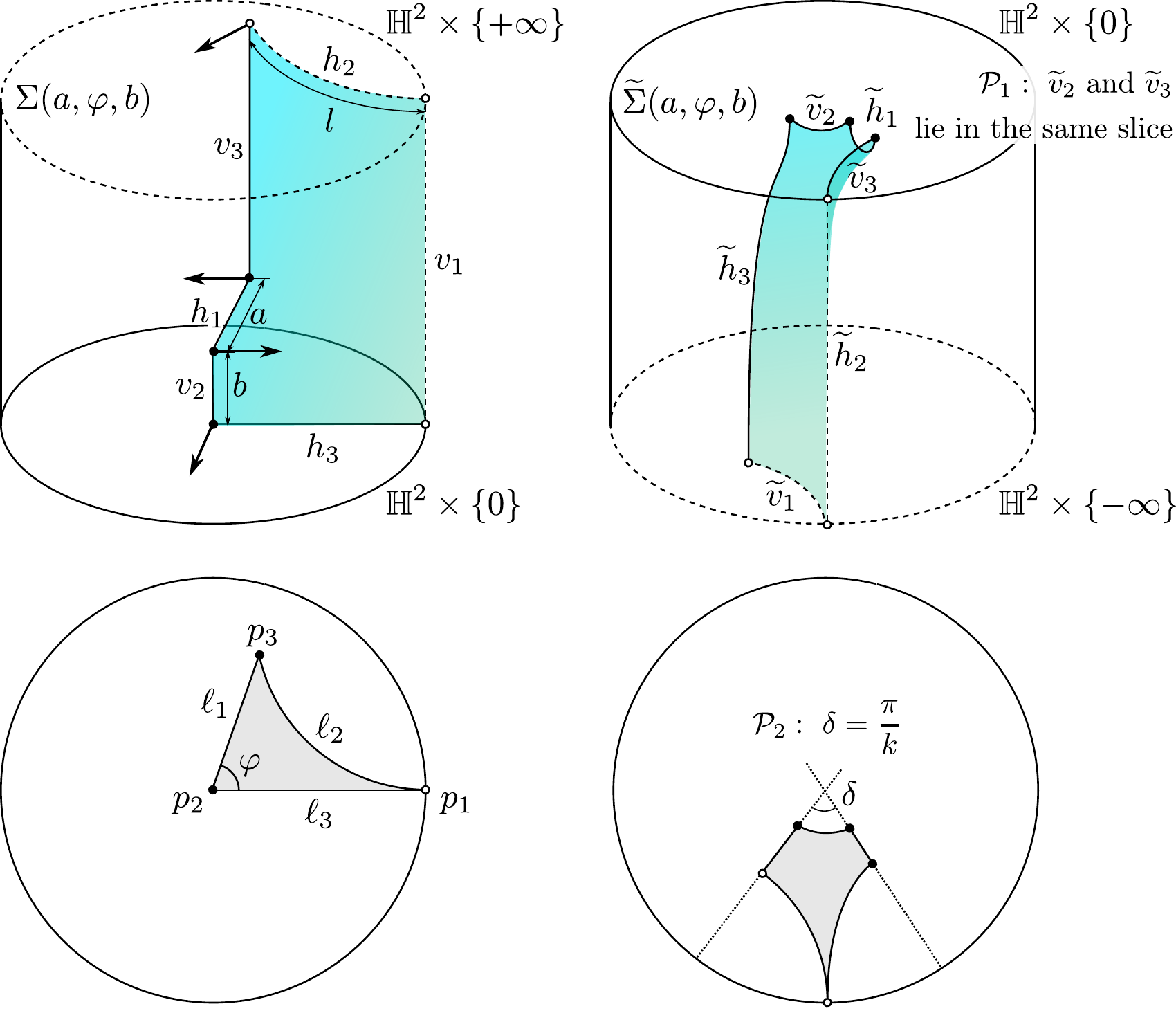}
	\caption{Conjugate surfaces $\Sigma(a,\varphi,b)$ and $\widetilde{\Sigma}(a,\varphi,b)$ and their domains $\Delta$ and $\widetilde\Delta$ in $\mathbb H^2$ in the case $l=\infty$. Dashed lines represent ideal geodesics, and white dots represent ideal vertexes. The arrows in $\Sigma(a,\varphi,b)$ represent the normal $N$ at the endpoints of $v_2$ and $v_3$, which rotates counterclockwise along both geodesics.  }\label{fig:conjugate-surfaces}
\end{figure}

We would like $\widetilde v_2$ and $\widetilde v_3$ to be contained in the same horizontal plane, as well as the complete geodesics of $\mathbb{H}^2$ containing $\widetilde h_1$ and $\widetilde h_3$ to intersect at an angle of $\frac{\pi}{k}$ (as shown in Figure~\ref{fig:conjugate-surfaces}, bottom right). Such a configuration would lead to the desired construction of saddle towers and $k$-noids, so we need to find values of $(a,\varphi,b)$ that solve the following period problems, inspired by the arguments in~\cite{Plehnert}:
\begin{enumerate}
\item \textbf{First period problem.} Let $\mathcal{P}_1(a,\varphi,b)$ be the difference of heights of the horizontal curves $\widetilde v_2$ and $\widetilde v_3$, i.e., the difference of heights of the endpoints of $\widetilde h_1$. Parameterizing $\widetilde{h}_1:[0,a]\to\mathbb{H}^2\times\mathbb{R}$ with $\widetilde h_1(0)\in\widetilde v_2$, $\widetilde h_1(a)\in\widetilde v_3$ and unit speed, by means of the properties of the conjugation, we can express
\begin{equation}\label{eqn:p1}
\mathcal{P}_1(a,\varphi,b)=\int_{\widetilde h_1}\langle\widetilde h_1',\partial_t\rangle=\int_{h_1}\langle \eta,\partial_t\rangle.
\end{equation}
where $\eta=-Jh_1'$ is the unit inward conormal vector to $\Sigma(a,\varphi,b)$ along $h_1$.

\item \textbf{Second period problem.} Let us work in the half-space model. After an ambient isometry, we can assume that $\widetilde{h}_3$ lies in the vertical plane $x=0$ and $\widetilde{v}_2:[0,b]\to\mathbb{H}^2\times\mathbb{R}$ is contained in the horizontal plane $t=0$ with endpoints $\widetilde{v}_2(0)\in\widetilde h_3$ and $\widetilde{v}_2(b)\in\widetilde h_1$. Expressing in coordinates $\widetilde v_2(t)=(x(t),y(t),0)$, we can also assume that $(x(0),y(0))=(0,1)$ and $x(t)<0$ when $t$ is close to $0$ (since $\widetilde v_2$ and $\widetilde h_3$ are orthogonal), see Figure~\ref{fig:horocycle-foliation}. We will write $(x(b),y(b))=(x_0,y_0)$ for simplicity. Recall that this parametrization comes (via conjugation) from the chosen orientation in which $v_2'=\partial_t$.

Let $\theta\in\mathcal C^\infty[0,b]$ be the angle of rotation of $\widetilde v_2$ with respect to the horocycle foliation in the sense of~\eqref{eqn:foliation-angle}, where we choose the initial angle $\theta(0)=\pi$. We will call $\theta_0=\theta(b)$ and assume in what follows that $\pi<\theta(t)<2\pi$ and $x(t)<0$ if $0<t\leq b$ (these inequalities will hold true by Lemma~\ref{lem:second-period}). The complete geodesic $\gamma\subset\mathbb{H}^2$ containing the projection of $\widetilde h_1$ can be parametrized as
\begin{equation}\label{eqn:gamma}
\gamma:(0,\pi)\to \mathbb{H}^2,\quad \gamma(t)=\left(x_0-y_0\dfrac{\cos(t)+\cos(\theta_0)}{\sin(\theta_0)}, -y_0 \dfrac{\sin(t)}{\sin(\theta_0)}\right).
\end{equation}
Note that $\gamma(\theta_0-\pi)=(x_0,y_0)$ and $\gamma'(\theta_0-\pi)=\frac{-1}{\sin(\theta_0)}(\sin(\theta_0)E_1-\cos(\theta_0)E_2)$. If $\gamma$ meets the $y$-axis at the point $\gamma(t_*)$, we define the second period as the cosine of the (non-oriented) angle $\delta$ at $\gamma(t_*)$ subtended by the arc $\widetilde v_2$ (see Figure~\ref{fig:horocycle-foliation}). From the parameterization~\eqref{eqn:gamma}, we can compute
\begin{equation}\label{eqn:p2}
\mathcal P_2(a,\varphi,b)=\cos(\delta)=\frac{\langle \gamma'(t_*),E_2\rangle}{|\gamma'(t_*)|}=\cos(t_*)=\frac{x_0\sin(\theta_0)}{y_0}-\cos(\theta_0).
\end{equation}
However, the right-hand side of~\eqref{eqn:p2} makes sense (and we will take it as the definition of $\mathcal P_2$) even though there is no such intersection point. Lemma~\ref{lem:second-period} will show that if $\mathcal P_1(a,\varphi,b)=0$ and $\mathcal P_2(a,\varphi,b)=\cos(\frac{\pi}{k})$ for some $k\geq 3$, then $\gamma$ and the positive $y$-axis do intersect with angle $\delta=\frac{\pi}{k}$. 
\end{enumerate}

\begin{figure}
	\includegraphics[height=5.8cm]{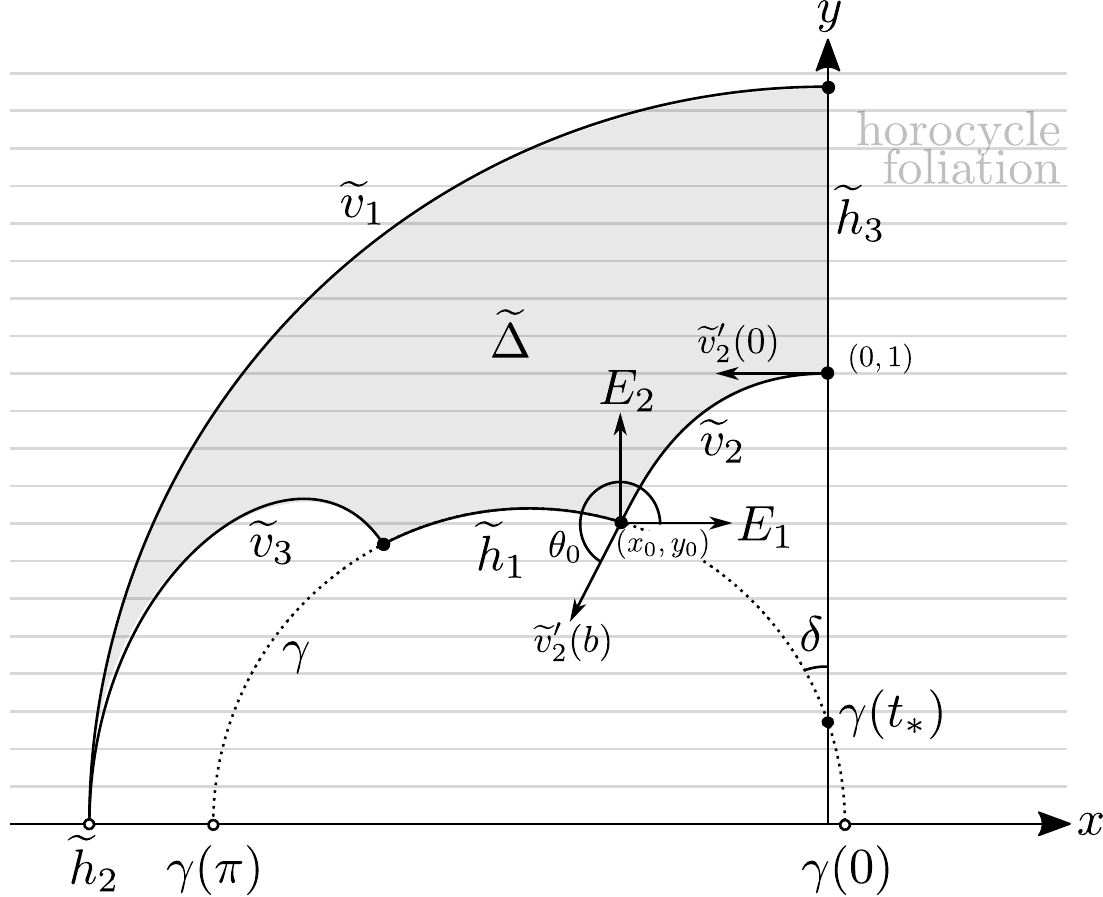}
	\caption{The angle $\theta_0$ of rotation of $\widetilde v_2$ with respect to the horocycle foliation at $\widetilde v_2(b)$, where we identify $\mathbb{H}^2\times\{0\}$ and $\mathbb H^2$. The surface $\widetilde\Sigma(a,\varphi,b)$ projects onto the shaded region $\widetilde\Delta$, with boundary the projections of the labeled curves. The complete geodesic $\gamma$ containing the projection of $\widetilde h_1$ appears in dotted line.}\label{fig:horocycle-foliation}
\end{figure}

The uniqueness of solution of the Jenkins--Serrin problem implies that $\widetilde\Sigma(a,\varphi,b)$ depends smoothly on the parameters $(a,\varphi,b)$ in the sense that, given a sequence $(a_n,\varphi_n,b_n)$ converging to some $(a,\varphi,b)$, the sequence of surfaces with boundary $\widetilde\Sigma(a_n,\varphi_n,b_n)$ converges in the $\mathcal{C}^k$-topology to $\widetilde\Sigma(a,\varphi,b)$ for all $k\geq 0$. This follows from standard convergence arguments for minimal graphs along with the continuity of the conjugation, see~\cite[Proposition~2.10]{MorRod} or~\cite[Proposition~2.3]{CMR}.

\subsection{Solving the period problems} In the sequel we will assume that $b$ is any non-negative real number and $(a,\varphi)$ lies in the domain
\[\Omega=\left\{(a,\varphi)\in\mathbb{R}^2:0<\varphi<\tfrac{\pi}{2},0<a<a_{\mathrm{max}}(\varphi)\right\}.\]
where
\begin{equation}\label{eqn:amax}
a_{\rm max}(\varphi)=2\arctanh\left(\tanh(l)\cos(\varphi)\right),
\end{equation}
and $\tanh(l)=1$ whenever $l=\infty$. The condition $0<a<a_{\rm max}(\varphi)$ means that the angle of $\Delta$ at $p_3$ is always greater than $\varphi$, and then an isosceles triangle $\Delta_{0}$ with vertexes $p_2$, $p_3$, $p_4$, and $d(p_2,p_4)=d(p_3,p_4)=l$ intersects $\Delta$ as in Figure~\ref{propl1} (the vertex $p_4$ is ideal provided that $l=\infty$). Note that $a_{\rm max}(\varphi)$ is the length of the unequal side of an isosceles triangle whose equal sides have length $l$ and whose equal angles are equal to $\varphi$, so formula~\eqref{eqn:amax} easily follows from the fact that the cosine of an angle of a hyperbolic right triangle is the quotient of the hyperbolic tangents of the adjacent side and the hypotenuse.

\begin{remark}
The restriction $(a,\varphi)\in\Omega$ is used in Lemma~\ref{lem:first-period} to compare $\Sigma(a,\varphi,b)$ with a solution of a Jenkins--Serrin problem over $\Delta_0$ and solve the first period problem. Similar arguments to those in Lemma~\ref{lem:first-period} show that if $a>a_{\mathrm{max}}(\varphi)$, then the first period problem has no solution, so the condition $(a,\varphi)\in\Omega$ is natural.
\end{remark}

\begin{figure}
		\includegraphics[height=4cm]{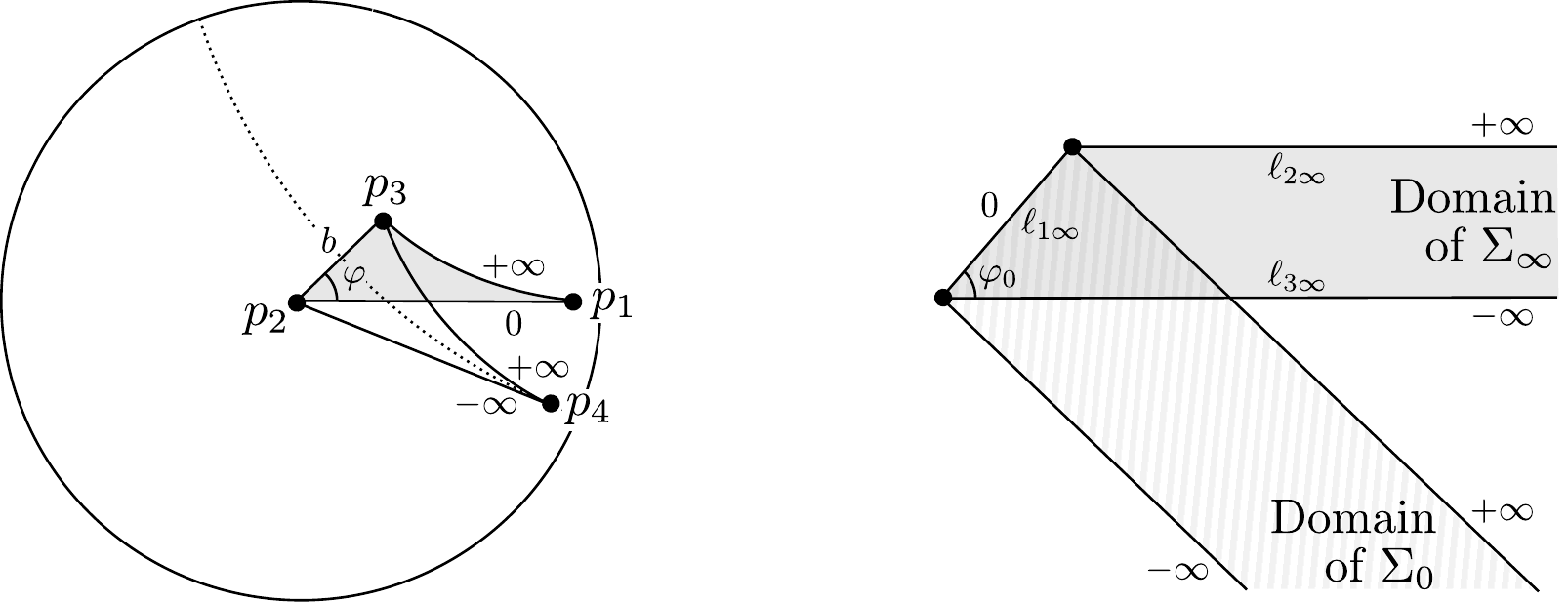}
		\caption{On the left, boundary values for Jenkins--Serrin problems in $\mathbb{H}^2$ solved by $\Sigma(a,\varphi,b)$ and $\Sigma_{0}(b)$, where the perpendicular bisector of $\ell_1$ is represented in dotted line and $l<\infty$. On the right, the limit $\Sigma_\infty\subset\mathbb{R}^3$ by rescaling (fixing the length of $\ell_1$ equal to $1$) and the helicoid $\Sigma_0\subset\mathbb{R}^3$ in the proof of Lemma~\ref{lem:first-period}.}\label{propl1}
	\end{figure}

\begin{lemma}\label{lem:p1-monotonicity}
	$\mathcal P_1:\Omega\times\mathbb{R}^+\to\mathbb{R}$ is a continuous and strictly decreasing function with respect to the third argument $b$.
\end{lemma}

\begin{proof}
Consider two surfaces $\Sigma_1=\Sigma(a,\varphi,b_1)$ and $\Sigma_2=\Sigma(a,\varphi,b_2)$ with $0<b_1<b_2$. Let us translate each $\Sigma_i$ vertically so that it takes the values $0$ along $\ell_1$ and $-b_i$ along $\ell_3$. The surface $\Sigma_1$ lies above $\Sigma_2$ and we can compare the vertical components of their inward-pointing conormals $\eta_1$ and $\eta_2$, which satisfy $\langle \eta_2 , \partial_t\rangle<\langle \eta_1, \partial_t\rangle$ in the interior of $h_1$ by the boundary maximum principle. In particular, 
\[\mathcal P_1(a,\varphi,b_2)=\int_{h_1}\langle \eta_2 , \partial_t\rangle<\int_{h_1}\langle \eta_1 , \partial_t\rangle=\mathcal P_1(a,\varphi,b_1).\qedhere\]
\end{proof}

In the proof of Lemma~\ref{lem:first-period} we need that a Jenkins--Serrin solution over $\Delta_0$ has radial limits as we approach a vertex of $\Delta_0$. To this end, we will state a more general result (we also point out that it extends easily to other $\mathbb{E}(\kappa,\tau)$-spaces).

\begin{lemma}\label{lem:radial-limits}
Let $\Omega\subset\mathbb{H}^2$ be an open domain with piecewise regular boundary, and assume that $\beta_1$ and $\beta_2$ are two regular components of $\partial\Omega$ that meet in a common vertex $p\in\partial\Omega$ with interior angle $0<\alpha<2\pi$. Suppose that $u\in\mathcal C^\infty(\Omega)$ is a solution to the minimal surface equation over $\Omega$ with bounded continuous limit values along $\beta_1$ and asymptotic value $+\infty$ or $-\infty$ along $\beta_2$. Then, $u$ has finite radial limits at $p$ along any geodesic segment interior to $\Omega$ and not tangent to $\beta_1$ or $\beta_2$.
\end{lemma}

\begin{proof}
Let $\gamma$ be the vertical geodesic segment projecting onto $p$ and lying in the boundary of $\Sigma$, the minimal surface spanned by $u$. Observe that $\Sigma$ extends analytically across $\gamma$ by axial symmetry, so the normal $N$ (with $\nu>0$ at the interior of $\Sigma$) extends smoothly to $\gamma$. Moreover, $N$ rotates monotonically along $\gamma$ because $\Sigma$ is a graph, as a consequence of the boundary maximum principle for minimal surfaces. Therefore, the conormal $J\gamma'=N\times\gamma'$ also rotates monotonically along $\gamma$ (see Figure~\ref{fig:orientation}). Since $J\gamma'$ is horizontal and tangent to the level curves of the height function of $\Sigma$, we deduce that the projections of such level curves form an open book foliation of a neighborhood of $p$ with binding at $p$. 

This implies that, when we approach $p$ along an interior geodesic $\sigma$ not tangent to $\beta_1$ or $\beta_2$, the limit of $u$ will be precisely the value of $u$ at the unique level curve (in the aforesaid foliation) tangent to $\sigma$ at $p$, so the desired limit exists and is finite.
\end{proof}

\begin{lemma}\label{lem:first-period}
There exists a unique function $f:\Omega\to\mathbb{R}_+$ such that $\mathcal P_1(a,\varphi,f(a,\varphi))=0$ for all $(a,\varphi)\in\Omega$. Furthermore,
\begin{enumerate}[label=\emph{(\alph*)}]
	\item $f$ is a continuous function;
	\item given $\varphi_0\in(0,\frac{\pi}{2})$, 
	\[\lim\limits_{a\to a_{\mathrm{max}}(\varphi_0)}f(a,\varphi_0)=+\infty,\qquad\lim\limits_{(a,\varphi)\to (0,\varphi_0)}f(a,\varphi)=0.\]
\end{enumerate}
\end{lemma}

\begin{proof}
Fix $(a,\varphi)\in\Omega$. Let $p_4$ be the point in the perpendicular bisector of the segment $\ell_1$ with $d(p_2,p_4)=d(p_3,p_4)=l$, such that the triangle $\Delta_0$ with vertexes $p_1$, $p_2$ and $p_4$ lies on the same side of $\ell_1$ as $\Delta$, see Figure~\ref{propl1}. Let $\Sigma_0(b)$ be the unique solution to the Jenkins--Serrin problem over the triangle $\Delta_0$ with values $b$ along the segment $\overline{p_2p_3}$, $+\infty$ along $\overline{p_3p_4}$, and $-\infty$ along $\overline{p_2p_4}$. Since $(a,\varphi)\in\Omega$, then $\Delta\cap\Delta_0\subset\mathbb{H}^2$ is a bounded geodesic triangle with vertexes $p_2$, $p_3$ and $q\in\ell_3$. Lemma~\ref{lem:radial-limits} says that $\Sigma_0(b)$ has a finite radial limit at $p_2$ along $\ell_3$ as a graph over $\Delta_0$. Therefore, if $b$ is large enough, then $\Sigma_0(b)$ is above $\Sigma(a,\varphi,b)$ over the boundary of $\Delta\cap\Delta_0$. By the maximum principle, it is also above $\Sigma(a,\varphi,b)$ in the interior of $\Delta\cap\Delta_0$. In particular, we get that $\Sigma\cap\Sigma_0(b)=h_1\cup v_2\cup v_3$ when $b$ is large. 

This means that can compare the vertical components of the inward-pointing conormals $\eta$ and $\eta_0$ of $\Sigma(a,\varphi,b)$ and $\Sigma_0(b)$, respectively, along the curve $h_1$ (as in Lemma~\ref{lem:p1-monotonicity}). By the boundary maximum principle for minimal surfaces, we get the strict inequality $\langle \eta , \partial_t\rangle < \langle \eta_0 , \partial_t\rangle$ in the interior of $h_1$, and hence  
\begin{equation}\label{prop:first-period-saddle-tower:eqn1}
	\mathcal P_1(a,\varphi,b)=\int_{h_1}\langle \eta , \partial_t\rangle<\int_{h_1}\langle \eta_0 , \partial_t\rangle=0,
\end{equation}
provided that $b$ is large enough. The last integral in~\eqref{prop:first-period-saddle-tower:eqn1} vanishes because $\Sigma_0(b)$ is axially symmetric with respect to the perpendicular bisector of $h_1$ in $\mathbb{H}^2\times\{b\}$.
	
Due to the continuity of $\Sigma(a,\varphi,b)$ with respect to the parameters $(a,\varphi,b)$, the surfaces $\Sigma(a,\varphi,b)$ converge to $\Sigma(a,\varphi,0)$ as $b\to 0$. We have that $\mathcal P_1(a,\varphi,0)>0$ since $\Sigma(a,\varphi,0)$ lies above the horizontal surface $\Delta\times\{0\}$ by the maximum principle, and we can compare the third coordinate of their conormals along the common boundary $h_1$ by the boundary maximum principle (note that the third coordinate of the conormal of $\Delta\times\{0\}$ identically vanishes). By the continuity and monotonicity of $\mathcal P_1$ with respect to $b$ proved in Lemma~\ref{lem:p1-monotonicity}, there exist a unique $b_0\in \mathbb{R}^+$ such that $\mathcal P_1(a,\varphi,b_0)=0$. Hence this defines unequivocally $f(a,\varphi)=b_0$. The continuity of $f$ is a consequence of its uniqueness. If $(a_n,\varphi_n)$ and $(a_n',\varphi_n')$ are two sequences in $\Omega$ converging to some $(a_\infty,\varphi_\infty)\in\Omega$ such that, after passing to a subsequence, $f(a_n,\varphi_n)\to b_\infty$ and $f(a_n',\varphi_n')\to b_\infty'$, and then $\mathcal P_1(a_\infty,\varphi_\infty,b_\infty)=\mathcal P_1(a_\infty,\varphi_\infty,b_\infty')=0$, whence $b_\infty=b_\infty'$, and item (a) is proved.
		
As for the first limit in item (b), assume by contradiction that there is a sequence $a_n\to a_{\mathrm{max}}(\varphi_0)$ such that $f(a_n,\varphi_0)$ converges, after passing to a subsequence, to some $b_\infty\in[0,+\infty)$. The surface $\Sigma_0(b_\infty)$ lies below $\Sigma(a_{\mathrm{max}}(\varphi_0),\varphi,b_\infty)$ as graphs over their common domain $\Delta=\Delta_0$ by maximum principle, because their boundary values are ordered likewise. Note that they have a common value $b_\infty$ along $\ell_1$, so their inward-pointing conormals can be compared along $h_1$ again by the boundary maximum principle. Since the $\Sigma_0(b_\infty)$ has zero period because of its symmetry, this contradicts the fact that $\Sigma(a_{\mathrm{max}}(\varphi_0),\varphi,b_\infty)$ also has zero period.
		
We will compute the limit as $(a,\varphi)$ approaches $(0,\varphi_0)$ again by contradiction, so let us assume that there is a sequence $(a_n,\varphi_n)$ tending to $(0,\varphi_0)$ such that (after passing to a subsequence) $f(a_n,\varphi_n)\to b_\infty$, with $b_\infty\in(0,+\infty]$. Let us translate the surfaces $\Sigma(a_n,\varphi_n,f(a_n,\varphi_n))$ vertically so that they take zero value along $\ell_1$ and $-f(a_n,\varphi_n)$ along $\ell_3$. Since $a_n\to 0$, we can blow up the surface and the metric of $\mathbb{H}^2\times\mathbb{R}$ in such a way $a_n$ is equal to $1$. The new sequence of rescaled surfaces converges in the $\mathcal C^k$-topology for all $k$ to a minimal surface $\Sigma_\infty$ in  Euclidean space $\mathbb{R}^3$. This surface $\Sigma_\infty$ is a graph over a domain of $\mathbb{R}^2$ bounded by three lines $\ell_{1\infty}$, $\ell_{2\infty}$ and $\ell_{3\infty}$ such that $\ell_{2\infty}$ and $\ell_{3\infty}$ are parallel and $\ell_{1\infty}$ makes an angle of $\varphi_0$ with $\ell_{2\infty}$. Moreover, $\Sigma_\infty$ takes values $+\infty$ along $\ell_{2\infty}$, $-\infty$ along $\ell_{3\infty}$ (since $b_\infty>0$), and $0$ along $\ell_{1\infty}$. Let us consider $\Sigma_0$ the helicoid of $\mathbb{R}^3$ with axis $\ell_{1\infty}$ which is a graph over a half-strip of $\mathbb{R}^2$ as depicted in Figure~\ref{propl1} (right). Since $0<\varphi_0<\frac{\pi}{2}$, the intersection of the domains of $\Sigma_0$ and $\Sigma_\infty$ is a triangle on whose sides the boundary values of $\Sigma_0$ are greater than or equal to the corresponding values of $\Sigma_\infty$. By the maximum principle, we deduce that $\Sigma_0$ lies above the surface $\Sigma_\infty$ also in the interior of that triangle. Hence, we can compare their conormals along $\ell_{1\infty}$ by the boundary maximum principle to conclude that the period of $\Sigma_\infty$ is not zero, which contradicts that each of the surfaces $\Sigma(a_n,\varphi_n,f(a_n,\varphi_n))$ has zero period.\qedhere
\end{proof}

This solves the first period problem, and we will now focus on the second one. To this end, we will use the notation defined in Section~\ref{subsec:periods} (see also Figure~\ref{fig:horocycle-foliation}).

\begin{lemma}\label{lem:second-period}
	Let $\varphi_0\in(0,\frac{\pi}{2})$ and $a\in(0,a_{\mathrm{max}}(\varphi_0))$.
	\begin{enumerate}[label=\emph{(\alph*)}]
		\item The inequalities $x(t)<0$ and $\pi<\theta(t)<2\pi$ hold true for all $t\in(0, b]$.
		\item If the curve $\gamma$ intersects the positive $y$-axis with angle $\delta$, then $\delta<\varphi_0$, in which case $\mathcal P_2(a,\varphi_0,f(a,\varphi_0))=\cos(\delta)$.
		\item If $\mathcal P_2(a,\varphi_0,f(a,\varphi_0))=\cos(\delta)$ for some $\delta\in(0,\varphi_0)$, then $\gamma$ intersects the positive $y$-axis with angle $\delta$.
		\item If $\mathcal P_2(a,\varphi_0,f(a,\varphi_0))=1$, then $\gamma$ and the $y$-axis are asymptotic geodesics intersecting at the ideal point $(0,0)$.
	\end{enumerate}
	Furthermore,
	\[\lim_{a\to 0}\mathcal P_2(a,\varphi_0,f(a,\varphi_0))=\cos(\varphi_0),\qquad 
	\lim_{a\to a_{\mathrm{max}}(\varphi_0)}\mathcal P_2(a,\varphi_0,f(a,\varphi_0))=+\infty.\]
\end{lemma}

\begin{proof}
We will identify $\widetilde v_2$ with its projection to $\mathbb H^2$ for the sake of simplicity. Therefore, $\widetilde v_2$ is strictly convex (in the hyperbolic geometry) towards the exterior of $\widetilde\Delta$ by Lemma~\ref{lem:orientation}, and this implies that any geodesic tangent to $\widetilde v_2$ lies locally in the interior of $\widetilde\Delta$ except for the point of tangency. In particular, we have that $\theta(t)>\pi$ for $t$ close to $0$ by just comparing $\widetilde v_2$ with the tangent geodesic at $\widetilde v_2(0)=(0,1)$ (see Figure~\ref{fig:gauss-bonnet}, left). Furthermore, if $\theta(t)>\pi$ does not hold for all $t\in(0,b]$, then at the smallest $t_0>0$ such that $\theta(t_0)=\pi$, the tangent geodesic has points outside $\widetilde\Delta$ arbitrarily close to $\widetilde v_2(t_0)$, which is a contradiction (see Figure~\ref{fig:gauss-bonnet}, left).

\begin{figure}
\begin{center}
\includegraphics[width=\textwidth]{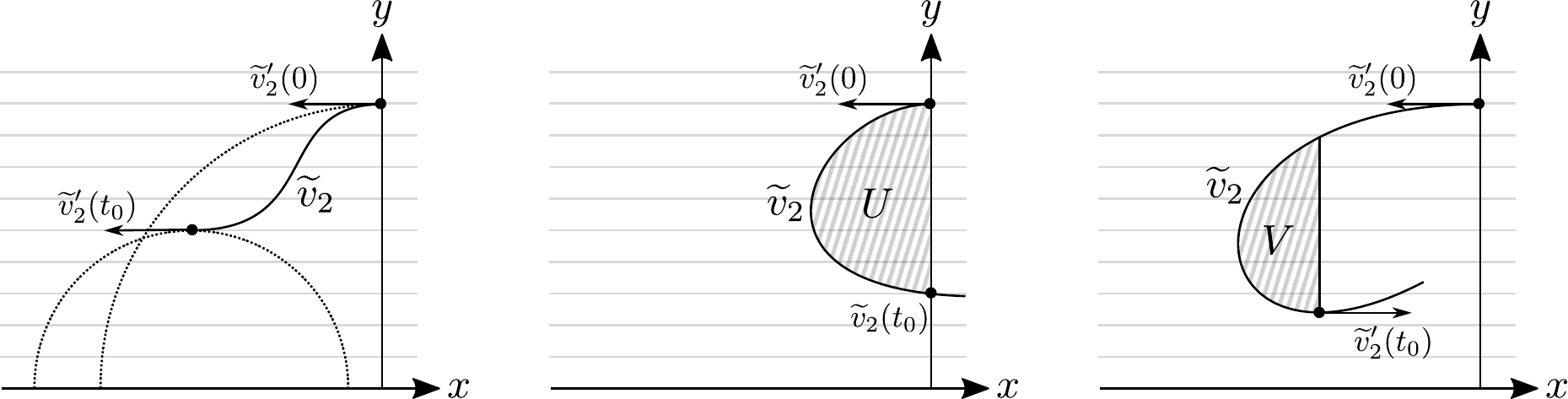}
\caption{Tangent geodesics at $\widetilde v_2(0)$ and at a first $t_0\in(0,b]$ such that $\theta(t_0)=\pi$ (left). A first $t_0\in(0,b]$ such that $x(t_0)=0$ (center). A first $t_0\in(0,b]$ such that $\theta(t_0)=2\pi$ (right). The domains $U$ and $V$ are those where we apply Gauss--Bonnet formula in Lemma~\ref{lem:second-period}.}\label{fig:gauss-bonnet}
\end{center}
\end{figure}

Assume by contradiction that $x(t)<0$ does not hold in general, and let $t_0\in(0,b]$ be the smallest value such that $x(t_0)=0$, so the curve $\widetilde v_2$ between $0$ and $t_0$ together with a segment of the $y$-axis enclose a bounded domain $U\subset\mathbb{H}^2$ (see Figure~\ref{fig:gauss-bonnet}, center). The curve $\widetilde v_2$ is convex towards the interior of $U$, and $U$ has two interior angles equal to $\frac\pi2$ and $\alpha=\theta(t_0)-\frac{3\pi}{2}\in(0,\pi)$, so Gauss--Bonnet formula yields 
\begin{equation}\label{lem:second-period:eqn1}
\begin{aligned}
0>-\area(U)&=2\pi+\int_{0}^{t_0}\kappa_g(t)\,\mathrm{d}t-(\tfrac{\pi}{2}+\pi-\alpha)\\
&>\tfrac{\pi}{2}+\int_{0}^{b}\kappa_g(t)\,\mathrm{d}t=\tfrac{\pi}{2}-\int_{0}^{b}\psi'(t)\,\mathrm{d}t=\tfrac{\pi}{2}-\varphi_0,
\end{aligned}\end{equation}
where $\kappa_g<0$ is the geodesic curvature with respect to the unit conormal $\widetilde N$ pointing outside $\widetilde\Delta$, see Lemma~\ref{lem:orientation}. We have also used that the angle $\psi$ of the normal $N$ along $v_2$, in the sense of~\eqref{eqn:rotation-angle}, rotates counterclockwise with $\psi'=-\kappa_g>0$ and $\psi(b)-\psi(0)=\varphi_0$. The inequality~\eqref{lem:second-period:eqn1} contradicts the assumption $\varphi_0\in(0,\frac\pi2)$.

Let us assume, again by contradiction, that there is (a first) $t_0\in(0,b]$ such that $\theta(t_0)=2\pi$. This implies that the normal geodesic to $\widetilde v_2$ at $t_0$ is a straight line parallel to the $y$-axis. Let $V\subset\mathbb{H}^2$ be the domain enclosed by this line together with an arc of $\widetilde v_2$ (see Figure~\ref{fig:gauss-bonnet}, right). Note that $V$ has two interior angles $\alpha\in(0,\pi)$ and $\frac{\pi}{2}$, plus $\widetilde v_2$ is convex towards $V$. Reasoning as in~\eqref{lem:second-period:eqn1}, we get the same contradiction $0>-\area(V)>\frac{\pi}{2}-\varphi_0$, which finishes the proof of item (a).

As for item (b), if $\gamma$ intersects the $y$-axis with angle $\delta$, then there is a region $W\subset\mathbb{H}^2$ bounded by $\gamma$, $\widetilde v_2$ and the $y$-axis. Gauss--Bonnet formula, in the same fashion as in Equation~\eqref{lem:second-period:eqn1}, gives the inequality $0>-\area(W)>\delta-\varphi_0$, which is equivalent to $\delta<\varphi_0$. The equality $\mathcal P_2(a,\varphi,f(a,\varphi))=\cos(\delta)$ was given in~\eqref{eqn:p2}.

We will now discuss items (c) and (d). Note that $\gamma(\pi)$ has negative first coordinate by the above analysis, so $\gamma$ intersects the $y$-axis if and only the first coordinate of
\begin{equation}\label{prop:second-period:eqn1}
\gamma(0)=\left(x_0-y_0\frac{1+\cos(\theta_0)}{\sin(\theta_0)},0\right)
\end{equation}
is positive (here, $\sin(\theta_0)<0$ because $\pi<\theta_0<2\pi$). If there exists $\delta\in(0,\varphi_0)$ such that $\mathcal P_2(a,\varphi_0,f(a,\varphi_0))=\frac{x_0\sin(\theta_0)}{y_0}-\cos(\theta_0)=\cos(\delta)\in (0,1)$, then the first coordinate in~\eqref{prop:second-period:eqn1} is positive, and it follows from (b) that the angle at the intersection is precisely $\delta$. If $\mathcal P_2(a,\varphi_0,f(a,\varphi_0))=1$, then the first coordinate of~\eqref{prop:second-period:eqn1} vanishes, so $\gamma$ and the $y$-axis are asymptotic at the ideal point $(0,0)$.

To finish the proof, let us analyze the limits. Integrating from $0$ to $b$ the identity $\theta'=\psi'-\cos(\theta)$ in Lemma~\ref{lem:orientation} (applied to $v_2$), and taking into account that $\theta(b)-\theta(0)=\theta_0-\pi$ and $\psi(b)-\psi(0)=\varphi_0$, we get the relation
\begin{equation}\label{eqn:twist2}
\theta_0=\varphi_0+\pi-\int_0^b\cos(\theta(s))\,\mathrm{d}s.
\end{equation}
In particular, $\theta_0\to\varphi_0+\pi$ and $(x_0,y_0)\to(0,1)$ as $b\to 0$ (note that the length of $\widetilde v_2$ goes to zero). This implies that the first component of~\eqref{prop:second-period:eqn1} is positive, i.e., $\gamma(0)$ and $\gamma(\pi)$ lie at distinct sides of the $y$-axis for $b$ small enough, so $\gamma$ intersects the positive $y$-axis at some point. By Lemma~\ref{lem:first-period}, if $a\in(0,a_{\mathrm{max}}(\varphi_0))$ tends to zero, then $b=f(a,\varphi_0)$ also tends to zero and 
\[\lim_{a\to 0}\mathcal P_2(a,\varphi_0,f(a,\varphi_0))=\lim_{a\to 0}\left(\frac{x_0\sin(\theta_0)}{y_0}-\cos(\theta_0)\right)=\cos(\varphi_0).\]
	
As for the limit $a\to a_{\mathrm{max}}(\varphi_0)$, let $(a_n,\varphi_0)\in\Omega$ be a sequence with $a_n\to a_{\mathrm{max}}(\varphi_0)$. Lemma~\ref{lem:first-period} tells us that $b_n=f(a_n,\varphi_0)\to +\infty$, so the surfaces $\Sigma(a_n,\varphi_0,b_n)$ converge, up to a subsequence and vertical translations (in such a way $h_1$ is a segment at height $0$) to a solution $\Sigma_\infty$ of a Jenkins--Serrin problem over an isosceles triangle with values $0$ along the unequal side and $+\infty$ and $-\infty$ along the other sides. We will denote in the sequel the elements of $\Sigma(a_n,\varphi_0,b_n)$ with a subindex $n$.
\begin{itemize}
	\item If $l<\infty$, the conjugate surfaces converge to $\widetilde\Sigma_\infty$, twice the fundamental piece of a symmetric saddle tower with four ends in the quotient (this conjugate construction is analyzed in~\cite{MorRod}). If we fix $\widetilde v_{2n}\subset\mathbb{H}^2\times\{0\}$, then the curves $\widetilde{v}_{1n}$ converge to a complete horizontal curve $\widetilde v_{1\infty}\subset\mathbb{H}^2\times\{-l\}$ (convex towards the exterior of the domain), and the curves $\widetilde h_{3n}$ tend to an ideal vertical segment $\widetilde h_{3\infty}$, see Figure~\ref{fig:catenoid}.	However, we will translate and rotate the surfaces first so that $\widetilde v_{2n}(0)=(0,1,0)$ and $\widetilde v_{2n}'(0)=-\partial_x$ in the half-space model in order to analyze the rotation $\theta_{0n}$ of $\widetilde v_{2n}'$ with respect to the horocycle foliation (i.e., we adapt the sequence to the setting of Figure~\ref{fig:horocycle-foliation}). This means that a subsequence of $\Sigma(a_n,\varphi_0,b_n)$ no longer converges to a saddle tower but to a subset of the vertical plane $x^2+y^2=1$. Therefore, $\theta_{0n}\to\frac{3\pi}{2}$ and $\widetilde v_{2n}(b_n)=(x_{0n},y_{0n})\to(-1,0)$ as $n\to\infty$. In view of~\eqref{prop:second-period:eqn1}, we deduce that $\gamma_n$ does not intersect the positive $y$-axis for large $n$, and~\eqref{eqn:p2} implies that $\mathcal P_2(a_n,\varphi_0,b_n)\to+\infty$.

	\item If $l=\infty$, then it is also well known~\cite{MorRod,Pyo} that the conjugate surfaces converge to $\widetilde\Sigma_\infty$, a quarter of a horizontal catenoid when we keep the point $\widetilde v_{2n}(b_n)$ fixed (and hence  the curves $\widetilde{v}_{1n}$ converge to a complete ideal horizontal geodesic $\widetilde v_{1\infty}\subset\mathbb{H}^2\times\{-\infty\}$). However, if we fix $\widetilde v_{2n}(0)=(0,1,0)$ and $\widetilde v_{2n}'(0)=-\partial_x$ instead, then a subsequence converges to a subset of the vertical plane $x^2+y^2=1$ as in the case $l<\infty$, so we can reason likewise.\qedhere
\end{itemize}
\end{proof}

\begin{figure}
\begin{center}
\includegraphics[width=0.8\textwidth]{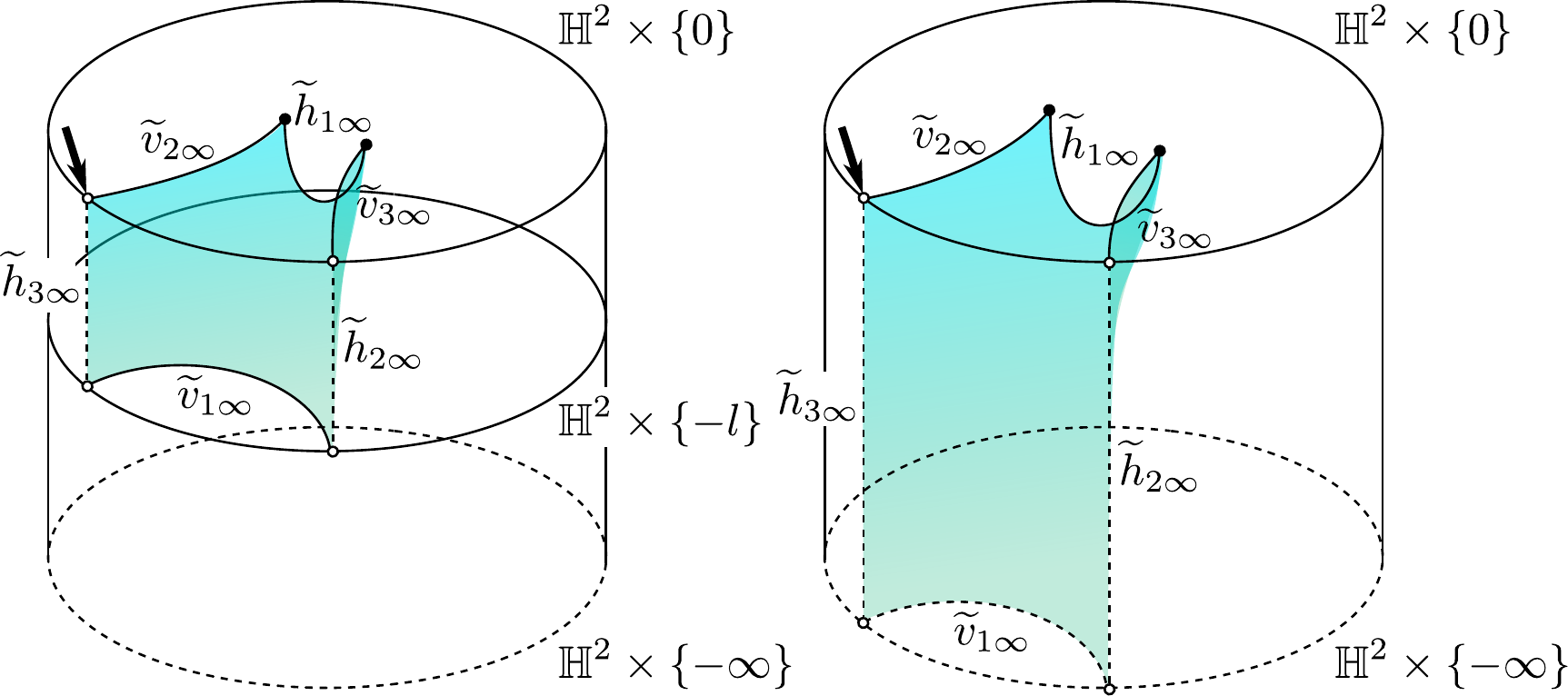}
\caption{The limit saddle tower ($l<\infty$) and catenoid ($l=\infty$) when $a\to a_{\rm max}(\varphi_0)$. In the proof Lemma~\ref{lem:second-period} we bring the points at which the arrows aim to a fixed point of $\mathbb H^2$, so we get vertical planes in the limit (instead of the saddle tower or the catenoid).}\label{fig:catenoid}
\end{center}
\end{figure}

\begin{proof}[Proof of Theorems~\ref{thm:knoids} and~\ref{thm:saddle-towers}]
Let $k\geq 3$. For each $\frac{\pi}{k}<\varphi<\frac{\pi}{2}$, Lemma~\ref{lem:second-period} ensures that $\mathcal P_2(a,\varphi,f(a,\varphi))$ tends to $\cos(\varphi)$ when $a\to 0$ and tends to $+\infty$ when $a\to a_{\mathrm{max}}(\varphi)$. Since $\cos(\varphi)<\cos(\frac{\pi}{k})$ and $\mathcal P_2$ is continuous, there exists some $a_\varphi\in(0,a_{\mathrm{max}}(\varphi))$ such that $\mathcal P_2(a,\varphi,f(a_\varphi,\varphi))=\cos(\frac{\pi}{k})$, though it might not be unique. Therefore, we deduce from item (c) of Lemma~\ref{lem:second-period} that $\widetilde\Sigma_\varphi=\widetilde\Sigma(a_\varphi,\varphi,f(a_\varphi,\varphi))$ solves both period problems. We will show that $\Sigma_\varphi=\Sigma(a_\varphi,\varphi,f(a,\varphi))$, and hence $\widetilde\Sigma_\varphi$, has finite total curvature by adapting Collin and Rosenberg's argument, see~\cite[Remark~7]{CR}.

To this end, we consider first the case $l=\infty$. For each $k\in\mathbb{N}$, let $p_1(k)\in\ell_3$ be such that $d(p_1(k),p_3)=k$, and let $\ell_2(k)$ (resp.\ $\ell_3(k)$) be the geodesic segment joining $p_3$ and $p_1(k)$ (resp.\ $p_2$ and $p_1(k)$). For each $n\in\mathbb{N}$ with $n\geq f(a_\varphi,\varphi)$, we will consider the Dirichlet problem over the triangle of vertexes $p_1(k)$, $p_2$ and $p_3$ with boundary values $f(a_\varphi,\varphi)$ on $\ell_1$, $n$ on $\ell_2(k)$ and $0$ over $\ell_3(k)$. These conditions span a unique compact minimal disk $\Sigma_\varphi^{k,n}$ (with boundary) which is a graph over the interior of the triangle. The surface $\Sigma_\varphi^{k,n}$ has geodesic boundary and six internal angles, all of them equal to $\frac{\pi}{2}$, so Gauss--Bonnet formula gives a total curvature of $-\pi$ for $\Sigma_\varphi^{k,n}$. As $k\to\infty$, the surfaces $\Sigma_\varphi^{k,n}$ converge uniformly on compact subsets (as graphs) to a surface $\Sigma_\varphi^n$ over $\Delta$ with boundary values $f(a_\varphi,\varphi)$ over $\ell_1$, $n$ over $\ell_2$, and $0$ over $\ell_3$ (this convergence is monotonic by the maximum principle). Therefore, Fatou's Lemma implies that the total curvature of $\Sigma_\varphi^{n}$ is at least $-\pi$. Finally, we let $n\to\infty$ so the $\Sigma_\varphi^{n}$ converge (also monotonically) to $\Sigma_\varphi$ on compact subsets, and the same argument implies that $\Sigma_\varphi$ has finite total curvature at least $-\pi$ (note that the Gauss curvature of a minimal surface in $\mathbb{H}^2\times\mathbb{R}$ is nowhere positive by Gauss equation). If $l<\infty$, the same idea works by just truncating at height $n$ (i.e., there is no need of introducing the sequence with index $k$).

By successive mirror symmetries across the planes containing the components of $\partial\widetilde\Sigma_\varphi$, we get a complete proper Alexandrov-embedded minimal surface $\overline\Sigma_\varphi\subset\mathbb{H}^2\times\mathbb{R}$.

\begin{itemize}
	\item  If $l=\infty$, then the curve $\widetilde v_1$ is an ideal horizontal geodesic, and we only need to reflect once about a horizontal plane, i.e., the plane containing $\widetilde v_2$ and $\widetilde v_3$. Hence, $\overline\Sigma_\varphi$ consists of $4k$ copies of $\widetilde\Sigma_\varphi$, so the total curvature in this case is not less than $-4k\pi$, and~\cite[Theorem~4]{HMR} ensures that $\overline\Sigma_\varphi$ is asymptotic to a certain geodesic polygon at infinity. From the above analysis, each end of $\overline\Sigma_\varphi$ has asymptotic boundary consisting of four complete ideal geodesics: two horizontal ones obtained from $\widetilde v_1$, and two vertical ones obtained from $\widetilde h_2$. Taking into account that $\overline\Sigma_\varphi$ has genus $g=1$, Equation~\eqref{eqn:generalized-GB} (with $m=k$) reveals that its total curvature is exactly $-4k\pi$. 

	Note that each end of $\overline\Sigma_\varphi$ is asymptotic to a vertical plane and it is contained in four copies of $\widetilde\Sigma_\varphi$. We claim that the subset of $\overline\Sigma_\varphi$ formed by these four copies is a symmetric bigraph, so the end is embedded in particular. This claim follows from the fact that two of these four pieces come from $\Sigma_\varphi$ and its axially symmetric surface with respect to $h_2$, which project to a quadrilateral of $\mathbb{H}^2$. Since this quadrilateral is convex, the Krust-type result in~\cite{HST} guarantees that the conjugate $\widetilde\Sigma_\varphi$ and its mirror symmetric surface across $\widetilde h_3$ form a graph. The other two copies needed to produce the aforesaid bigraph are their symmetric ones with respect to the slice containing $\widetilde v_2$ and $\widetilde v_3$.

	\item If $l<\infty$, then the composition of the reflections with respect to the horizontal planes containing $\widetilde v_1$ and $\widetilde v_3$ is a vertical translation $T$ of length $2l$. Thus, $\overline\Sigma_\varphi$ induces a surface in the quotient of $\mathbb{H}^2\times\mathbb{R}$ by $T$ with total Gauss curvature at least $-4k\pi$, since it consists of $4k$ pieces isometric to $\widetilde\Sigma_\varphi$. This surface has genus $1$ and $2k$ ends, so it follows from the main theorem in~\cite{HM} that its total curvature is exactly $-4k\pi$. This result also implies that each end of $\overline\Sigma_\varphi$ is asymptotic to a vertical plane (in the quotient). \qedhere
\end{itemize}
\end{proof}

 \begin{remark}
 It is important to notice that we have not proved the uniqueness of the surface $\Sigma_\varphi$. This would be automatically true if we could show that the second period $\mathcal P_2(a,\varphi,f(a,\varphi))$ is strictly increasing in the parameter $a$, though a comparison of the surfaces for different values of $a$ seems to be difficult, since we do not even know if the function $f$ solving the first period problem is monotonic.
 \end{remark}

As $\varphi$ approaches $\frac{\pi}{k}$, the value $a_\varphi$ solving the two period problems in the proof of Theorems~\ref{thm:knoids} and~\ref{thm:saddle-towers} goes to zero, and the surface $\widetilde\Sigma_\varphi$ converges, after rescaling, to a genus $1$ minimal $k$-noid in $\mathbb{R}^3$ (as in item (b) of Lemma~\ref{lem:first-period}). Moreover, when $\varphi$ approaches $\frac{\pi}{2}$, the surface $\Sigma_\varphi$ converges to an open subset of a helicoid in $\mathbb{R}^3$ after rescaling, and it follows that the conjugate surfaces $\widetilde\Sigma_\varphi$ must converge a quarter of a catenoid in $\mathbb{R}^3$ (the curve $\widetilde h_1$ converges to half of the neck of such catenoid).

\subsection{The embeddedness problem}

In the proof of Theorems~\ref{thm:knoids} and~\ref{thm:saddle-towers}, it is shown that the conjugate piece $\widetilde{\Sigma}_\varphi$ is a graph over the domain $\widetilde \Delta\subset\mathbb{H}^2$. But it could happen that when we reflect $\widetilde{\Sigma}_\varphi$ over the vertical plane containing $\widetilde h_1$, the resulting surface is not embedded since the reflected curve of $\widetilde v_3$ might intersect $\widetilde v_3$. Observe that, as the family of examples with $k$ ends converges to a genus $1$ minimal $k$-noid in $\mathbb{R}^3$ after blow up, see also~\cite{Maz}, there do exist non-embedded examples of $k$-noids and saddle towers with genus $1$ in $\mathbb{H}^2\times\mathbb{R}$ for all $k\geq 3$.

\begin{figure}
	\includegraphics[width=\textwidth]{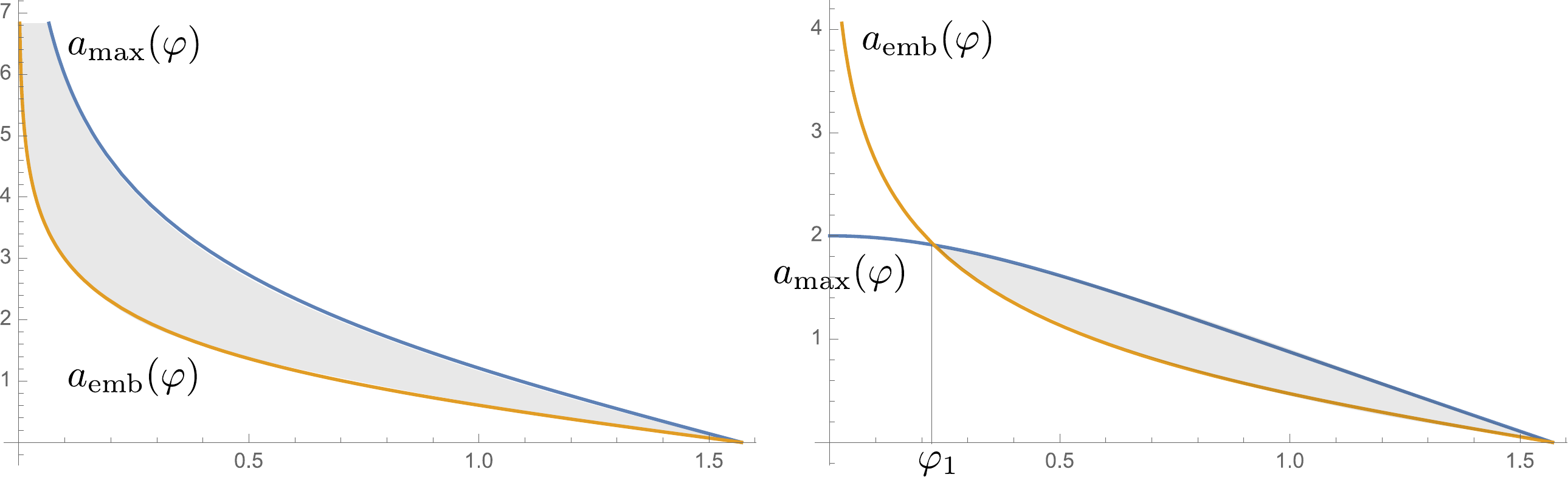}
	\caption{Graphics of the functions $\varphi\mapsto a_{\text{max}}(\varphi)$ and $\varphi\mapsto a_{\text{emb}}(\varphi)$ with $l=\infty$ (left) and $l=1$ (right). In the shaded regions, embeddedness is guaranteed by the Krust property.}\label{aemb} 
\end{figure}

Therefore, embeddedness is guaranteed if the extended surface by reflection about the vertical plane containing $\widetilde h_1$ is embedded. The Krust property yields this if the initial surface $\Sigma_\varphi$ extended by axial symmetry about the geodesic $h_1$ is still a graph over a convex domain, i.e., if the angle of $\Delta$ at $p_3$ is at most $\frac{\pi}{2}$. Elementary hyperbolic geometry shows that this is equivalent to $a\geq a_{\text{emb}}(\varphi)$, where
	\begin{equation}\label{eqn:emb}
		 a_{\text{emb}}(\varphi)=\arcsinh(\tanh(l)\cot(\varphi)),
	\end{equation}
and $\tanh(l)=1$ if $l=\infty$. Hence, the surfaces in Theorems~\ref{thm:knoids} and~\ref{thm:saddle-towers} are properly embedded provided that $a_{\mathrm{emb}}(\varphi)\leq a_\varphi<a_{\text{max}}(\varphi)$. If $l=\infty$, then $a_{\text{emb}}(\varphi)<a_{\text{max}}(\varphi)$ for all $\varphi\in(0,\frac\pi2)$; if $l<\infty$, there exists $\varphi_1\in(0,\frac\pi2)$  such that $a_{\text{emb}}(\varphi)<a_{\text{max}}(\varphi)$ if and only if $\varphi\in(\varphi_1,\frac\pi2)$, see Figure~\ref{aemb}.

However, on the one hand, it seems difficult to know if a value of $a_\varphi$ solving both period problems lies in this interval; on the other hand, embeddedness may occur even if $a\geq a_{\rm emb}(\varphi)$ does not hold. It is expected that there are always values of $(a,\varphi)$ producing embedded examples solving the two period problems for all $k\geq 3$, and it seems reasonable that this occurs when $\varphi$ becomes close to $\frac{\pi}{2}$.

\subsection{Examples with infinitely many ends.}
Let us tackle the proof of Theorem~\ref{thm:infty-noids}, which is a particular case of the above constructions for $l=\infty$ (the proof can be easily adapted to the case $l<\infty$). Lemma~\ref{lem:second-period} and the continuity of $\mathcal P_2$ imply that, for all $\varphi\in(0,\frac\pi2)$, there are values of $a\in(0,a_{\mathrm{max}}(\varphi))$ such that $\mathcal P_2(a,\varphi,f(a,\varphi))$ is either equal to $1$ or greater than $1$. Let us study these two cases:
\begin{itemize}
	\item If $\mathcal P_2(a,\varphi,f(a,\varphi))=1$, then $\widetilde h_1$ and $\widetilde h_3$ are contained in vertical planes over asymptotic geodesics of $\mathbb{H}^2$ in view of item (d) of Lemma~\ref{lem:second-period}. This means that $\widetilde\Sigma(a,\varphi,f(a,\varphi))$ is contained in the region between these two geodesics, see Figure~\ref{fig:infty-noids} (center), and mirror symmetries across the corresponding vertical planes span a group of isometries fixing the common point at infinity. This group contains a discrete group of parabolic translations, and gives rise to the $1$-parameter family of parabolic $\infty$-noids.
	\item The case $\mathcal P_2(a,\varphi,f(a,\varphi))>1$ occurs in an open subset of $\Omega$, and gives rise to the $2$-parameter family of hyperbolic $\infty$-noids. The two geodesics of $\mathbb{H}^2$ containing the projections of $\widetilde h_1$ and $\widetilde h_3$ do not intersect in this case and successive reflections across their associated vertical planes span a group of isometries containing a discrete group of hyperbolic translations, see Figure~\ref{fig:infty-noids} (right).
\end{itemize}

\begin{figure}[t]
	\includegraphics[width=\textwidth]{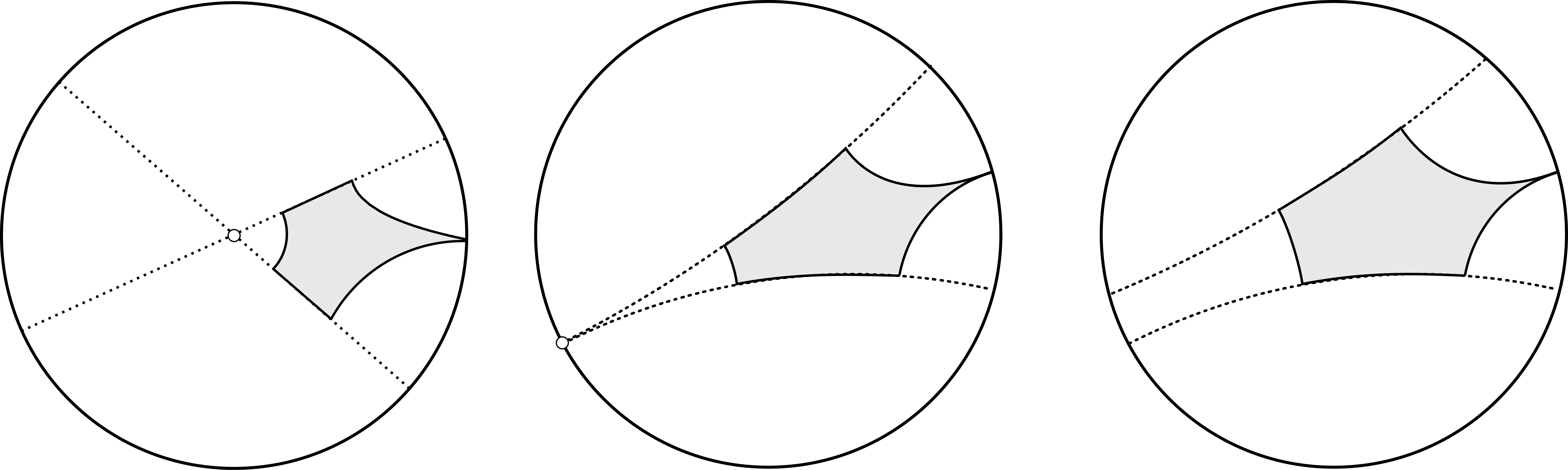}
	\caption{The fundamental domains of a $3$-noid (left), a parabolic $\infty$-noid (middle), and a hyperbolic $\infty$-noid (right). Dotted curves represent geodesics containing the projection of $\widetilde h_1$ and $\widetilde h_3$.} \label{fig:infty-noids}
\end{figure}

Similar arguments to those in the proof of Theorems~\ref{thm:knoids} and~\ref{thm:saddle-towers} (using the description of periodic surfaces with finite total curvature in~\cite{HM}) show that each end of the constructed surfaces is embedded and has finite total curvature, plus the global surface is Alexandrov-embedded. Observe that in the case of hyperbolic $\infty$-noids, we can always choose $a\geq a_{\mathrm{emb}}(\varphi)$, defined in the previous section, which means that whenever the parameters $(a,\varphi)$ lie in this open subset of $\Omega$, the reflected surface is a properly embedded hyperbolic $\infty$-noid. In the case of parabolic $\infty$-noids, we are not able to guarantee global embeddedness.

\medskip

\noindent\textbf{Competing interests:} The authors declare none.


\begin{thebibliography}{OO}

\bibitem{CMR} J.\ Castro-Infantes, J.\ M.\ Manzano, M.\ Rodríguez.
\newblock A construction of constant mean curvature surfaces in $\mathbb{H}^2\times\mathbb{R}$ and the Krust property.
\newblock Preprint available at \href{https://arxiv.org/abs/2012.13192}{arXiv:2012.13192}.

\bibitem{CR} P.\ Collin,  H.\ Rosenberg.
\newblock Construction of harmonic diffeomorphisms and minimal graphs.
\newblock \emph{Ann.\ of Math.\ (2)}, \textbf{172} (2010), no.\ 3, 1879--1906.

\bibitem{Dan} B.\ Daniel. 
\newblock Isometric immersions into $\mathbb{S}^n\times\mathbb{R}$ and $\mathbb{H}^n\times\mathbb{R}$ and applications to minimal surfaces.
\newblock \emph{Trans.\ Amer.\ Math.\ Soc.}, \textbf{361} (2009), no.\ 12, 6255--6282.

\bibitem{HM} L.\ Hauswirth, A.\ Menezes.
\newblock On doubly periodic minimal surfaces in $\mathbb H^2\times\mathbb R$ with finite total curvature in the quotient space.
\newblock \emph{Ann.\ Mat.\ Pura Appl.}, \textbf{195} (2016), no.\ 5, 1491--1512.


\bibitem{HMR} L.\ Hauswirth, A.\ Menezes, M.\ Rodríguez.
\newblock On the characterization of minimal surfaces with finite total curvature in $\mathbb{H}^2\times\mathbb{R}$ and $\widetilde{PSL}_2(\mathbb{R})$.
\newblock \emph{Calc.\ Var.}, \textbf{58} (2019), no.\ 2, Art 80, 24 pp.

\bibitem{HNST} L.\ Hauswirth, B.\ Nelli, R.\ Sa Earp, E.\ Toubiana.
\newblock Minimal ends in $\mathbb{H}^2\times\mathbb{R}$ with finite total curvature and a Schoen type theorem.
\newblock \emph{Adv.\ Math.}, \textbf{274} (2015), 199--240.

\bibitem{HR} L.\ Hauswirth, H.\ Rosenberg.
\newblock Minimal surfaces of finite total curvature in $\mathbb{H}^2\times\mathbb{R}$. 
\newblock \emph{Mat.\ Contemp.}, \textbf{31} (2006), 65--80. 

\bibitem{HST} L.\ Hauswirth, R.\ Sa Earp, E.\ Toubiana.
\newblock Associate and conjugate minimal immersions in $M\times\mathbb{R}$.
\newblock\emph{Tohoku Math.\ J.}, \textbf{60} (2008), 267--286.

\bibitem{ManTor} J.\ M.\ Manzano, F.\ Torralbo.
\newblock New examples of constant mean curvature surfaces in $\mathbb{S}^2\times\mathbb{R}$ and $\mathbb{H}^2\times\mathbb{R}$.
\newblock \emph{Michigan Math.\ J.}, \textbf{63} (2014), no.\ 4, 701--723.

\bibitem{ManTor2} J.\ M.\ Manzano, F.\ Torralbo.
\newblock Compact embedded surfaces with constant mean curvature in $\mathbb{S}^2\times\mathbb{R}$.
\newblock \emph{Amer.\ J.\ Math.}, \textbf{142} (2020), no.\ 4, 1981--1994.

\bibitem{MPT} J.\ M.\ Manzano, J.\ Plehnert, F.\ Torralbo.
\newblock Compact embedded minimal surfaces in $\mathbb{S}^2\times\mathbb{S}^1$
\newblock \emph{Comm.\ Anal.\ Geom.}, \textbf{24} (2016), no.\ 2, 409--429.

\bibitem{MMR} F.\ Martín, R.\ Mazzeo, M.\ Rodríguez.
\newblock Minimal surfaces with positive genus and finite total curvature in $\mathbb{H}^2\times\mathbb{R}$.
\newblock \emph{Geom.\ Top.}, \textbf{18} (2014), 141--177.

\bibitem{Maz} L.\ Mazet.
\newblock The Plateau problem at infinity for horizontal ends and genus 1.
\newblock \emph{Indiana Univ.\ Math.\ J.}, \textbf{55} (2006), no.\ 1, 15--64.

\bibitem{MRR} L.\ Mazet, M.\ Rodríguez, H.\ Rosenberg.
\newblock The Dirichlet problem for the minimal surface equation --with possible infinite boundary data-- over domains in a Riemannian surface.
\newblock \emph{Proc. London Math. Soc.}, \textbf{102} (2011), no.\ 3, 985--1023.

\bibitem{MRR2} L.\ Mazet, M.\ Rodríguez, H.\ Rosenberg.
\newblock Periodic constant mean curvature surfaces in $\mathbb{H}^2\times\mathbb{R}$.
\newblock \emph{Asian J. Math.}, \textbf{18} (2014), no.\ 5, 829--858.

\bibitem{MorRod} F.\ Morabito, M.\ Rodríguez.
\newblock Saddle Towers and minimal $k$-noids in $\mathbb{H}^2\times\mathbb{R}$.
\newblock \emph{J. Inst. Math. Jussieu}, \textbf{11} (2012), no.\ 2, 333--349.

\bibitem{NR} B.\ Nelli, H.\ Rosenberg.
\newblock Minimal surfaces in $\mathbb{H}^2\times\mathbb{R}$.
\newblock\emph{Bull. Braz. Math. Soc.}, \textbf{33} (2002), no.\ 2, 263--292.

\bibitem{Plehnert} J.\ Plehnert.
\newblock Surfaces with constant mean curvature $\frac12$ and genus one in $\mathbb{H}^2\times\mathbb{R}$.
\newblock Preprint available at \href{https://arxiv.org/abs/1212.2796}{arXiv:1212.2796}.

\bibitem{Plehnert2} J.\ Plehnert.
\newblock Constant mean curvature $k$-noids in homogeneous manifolds.
\newblock \emph{Illinois J.\ Math.}, \textbf{58} (2014), no.\ 1, 233--249.

\bibitem{Pyo} J.\ Pyo.
\newblock New complete embedded minimal surfaces in $\mathbb{H}^2\times\mathbb{R}$.
\newblock \emph{Ann.\ Glob.\ Anal.\ Geom.}, \textbf{40} (2011), no.\ 2, 167--176.

\bibitem{PR} J.\ Pyo, M.\ Rodríguez.
\newblock Simply connected minimal surfaces with finite total curvature in $\mathbb{H}^2\times\mathbb{R}$.
\newblock \emph{Int. Math. Res. Not.}, \textbf{2014} (2014), no.\ 11, 2944--2954.

\end{thebibliography}
\end{document}